\documentclass[12pt,reqno,openbib,runningheads,a4paper,portrait]{amsart}%
\usepackage{lineno}
\usepackage[authoryear,sectionbib]{natbib}
\usepackage[bookmarks=true,colorlinks=true,linkcolor=blue,citecolor=blue,a4paper]%
{hyperref}
\usepackage[labelformat=simple,labelfont={bf},textfont={up},labelsep=period]%
{caption}
\usepackage{amsfonts}
\usepackage{amsmath}
\usepackage{amssymb}
\usepackage{graphicx}%
\providecommand{\U}[1]{\protect\rule{.1in}{.1in}}
\providecommand{\U}[1]{\protect\rule{.1in}{.1in}}
\newtheorem{theorem}{Theorem}[section]

\newtheorem{proposition}{Proposition}[section]

\newtheorem{lemma}{Lemma}[section]

\makeatletter
\renewcommand{\@biblabel}[1]{}
\makeatother
\begin{document}

\begin{center}
{\Large Robust and Smooth Estimation of the Extreme Tail Index via Weighted
Minimum Density Power Divergence}\bigskip

{\large Saida Mancer, Abdelhakim Necir}$^{\ast},$ {\large Djamel
Meraghni}\medskip\\[0pt]

{\small \textit{Laboratory of Applied Mathematics, Mohamed Khider University,
Biskra, Algeria}}\medskip\medskip
\end{center}

\noindent By introducing a weight function into the density power divergence,
we develop a new class of robust and smooth estimators for the tail index of
Pareto-type distributions, offering improved efficiency in the presence of
outliers. These estimators can be viewed as a robust generalization of both
weighted least squares and kernel-based tail index estimators. We establish
the consistency and asymptotic normality of the proposed class. A simulation
study is conducted to assess their finite-sample performance in comparison
with existing methods. \medskip\medskip

 \medskip

\noindent\textbf{AMC 2020 Subject Classification:} 62G32; 62G05; 62G20; 62G35.

\vfill

\vfill

\noindent{\small $^{\text{*}}$Corresponding author:
\texttt{ah.necir@univ-biskra.dz}\newline\noindent\textit{E-mail address:}%
\newline\texttt{mancer.saida731@gmail.com} (S.~Mancer)\newline%
\texttt{djamel.meraghni@univ-biskra.dz} (D.~Meraghni)}

\section{\textbf{Introduction\label{sec1}}}

\noindent Let $X_{1},...,X_{n}$ be independent and identically distributed
(iid) of non-negative random variables (rv's) viewed as $n$ copies of a rv
$X,$ defined over some probability space $\left(  \Omega,\mathcal{A}%
,\mathbf{P}\right)  ,$ with cumulative distribution function (cdf) $F.$ We
assume that the tail distribution $\overline{F}:=1-F$ is regularly varying at
infinity\ with negative index $\left(  -1/\gamma\right)  ,$ i.e., for every
$x>0,$%
\begin{equation}
\frac{\overline{F}\left(  ux\right)  }{\overline{F}\left(  u\right)
}\rightarrow x^{-1/\gamma},\text{ as }u\rightarrow\infty. \label{RV}%
\end{equation}
The parameter $\gamma>0$ often referred to as the shape parameter, tail index,
or extreme value index (EVI), plays a fundamental role in extreme value
analysis, as it characterizes the heaviness of the distribution's right tail.
The estimation of $\gamma$ has received considerable attention over the past
four decades. The most widely used estimator of $\gamma$ is Hill's estimator
\citep{Hill75} defined as%
\[
\widehat{\gamma}_{k}^{\left(  H\right)  }:=\frac{1}{k}%
{\displaystyle\sum\limits_{i=1}^{k}}
\log\frac{X_{n-i+1:n}}{X_{n-k:n}}=%
{\displaystyle\sum\limits_{i=1}^{k}}
\frac{i}{k}\log\frac{X_{n-i+1:n}}{X_{n-i:n}},
\]
where $X_{1:n}\leq...\leq X_{n:n}$ denote the order statistics pertaining to
the sample $X_{1},...,X_{n}$ and $k=k_{n}$ is an integer sequence satisfying
$1<k<n,$ $k\rightarrow\infty$ and $k/n\rightarrow0$ as $n\rightarrow\infty.$
The discrete nature and lack of stability of Hill's estimator constitute
significant drawbacks. In particular, the inclusion of a single additional
upper-order statistic, i.e. increasing $k$ by one, can lead to substantial
deviations from the true value of the tail index, reflecting the estimator's
sensitivity to the choice of the threshold $k.$ As a result, plotting this
estimator as a function of the number of upper-order statistics typically
produces a zigzag pattern, which leads to instability and fluctuations in
determining the optimal sample fraction used in the estimation of the
computation of \textbf{\ }$\widehat{\gamma}_{k}^{\left(  H\right)  }%
\mathbf{.}$\textbf{\ }To address this issue, \cite{CDM85} (CDM) introduced
more general weights in place of the natural weight $i/k$ appearing in the
second expression of Hill's estimator $\widehat{\gamma}_{k}^{\left(  H\right)
},$ leading to the definition of the following kernel-type estimator:
\begin{equation}
\widehat{\gamma}_{k,\mathbb{K}}^{\left(  CDM\right)  }:=%
{\displaystyle\sum_{i=1}^{k}}
\dfrac{i}{k}\mathbb{K}\left(  \dfrac{i}{k+1}\right)  \log\dfrac{X_{n-i+1:n}%
}{X_{n-i:n}}, \label{cdm}%
\end{equation}
where $\mathbb{K}$ is a continuous nonnegative nonincreasing function on
$\left(  0,1\right)  $ such that $\int_{0}^{1}\mathbb{K}\left(  s\right)
ds=1.$ For their part, \cite{HLM2006} used the weighted least squares
estimator (WLSE) given by%
\begin{equation}
\widehat{\gamma}_{k,J}:=\left(  \beta k\right)  ^{-1}%
{\displaystyle\sum\limits_{i=1}^{k}}
J\left(  \dfrac{i}{k+1}\right)  \log\dfrac{X_{n-i+1:n}}{X_{n-k:n}},
\label{HLM}%
\end{equation}
where $\beta:=-\int_{0}^{1}J\left(  s\right)  \log sds$ and $J$ is a suitable
continuous nonnegative nonincreasing function defined on $\left(  0,1\right)
$ and vanishes elsewhere such that $\int_{0}^{1}J\left(  s\right)  ds=1.$
Similar estimators to $\widehat{\gamma}_{k,J}$ are also considered in
\cite{BVT96}, \cite{V99} and recently \cite{CMS21}. The authors pointed out
that the least squares estimator $\widehat{\gamma}_{k,J}$ may be rewritten
into CDM's one $\widehat{\gamma}_{k,\mathbb{K}}^{\left(  CDM\right)  }$ for
the kernel function%
\begin{equation}
\mathbb{K}\left(  s\right)  =\left(  s\beta\right)  ^{-1}\int_{0}^{s}J\left(
t\right)  dt,\text{ for }s\in\left(  0,1\right)  . \label{f}%
\end{equation}
For example, by choosing $J_{\log}\left(  s\right)  :=-\mathbb{I}_{\left(
0,1\right)  }\left(  s\right)  \log s,$ we obtain
\[
\mathbb{K}\left(  s\right)  =\mathbb{I}_{\left(  0,1\right)  }\left(
s\right)  \left(  1-\log s\right)  .
\]
However, the converse does not necessarily hold; that is $\mathbb{K}$ may be a
kernel function without $J$ being one. For instance, if we take $\mathbb{K}%
\left(  s\right)  :=-\mathbb{I}_{\left(  0,1\right)  }\left(  s\right)  \log
s,$ then the corresponding function $J\left(  s\right)  $ is identically zero.
Here, $\mathbb{I}_{A}$ denotes the indicator function of the set $A.$ The most
commonly used weighted functions $J:$ the indicator function $J_{0}%
:=\mathbb{I}_{\left(  0,1\right)  },$ as well as the Epanechnikov, biweight,
triweigh and quadweight weighted functions which are defined on the interval
$0<s<1$ by%
\begin{equation}%
\begin{array}
[c]{rl}%
J_{1}\left(  s\right)  :=\dfrac{2}{3}\left(  1-s^{2}\right)  , & J_{2}\left(
s\right)  :=\dfrac{15}{8}\left(  1-s^{2}\right)  ^{2},\medskip\\
J_{3}\left(  s\right)  :=\dfrac{35}{16}\left(  1-s^{2}\right)  ^{3}, &
J_{4}\left(  s\right)  :=\dfrac{315}{128}\left(  1-s^{2}\right)  ^{4},
\end{array}
\label{J}%
\end{equation}
and zero elsewhere respectively. For the use of this type of weight functions,
we refer the reader to \cite{GLW2003} and \cite{HLM2006}.\medskip

\noindent The widely adopted approach for estimating the parameters of an
extreme value distribution in extreme value analysis is the maximum likelihood
estimation (MLE) method. While MLEs enjoy desirable asymptotic properties,
they can be highly sensitive to outliers or deviations from the assumed
extreme value models \cite{BS2000}. Robust statistics offer a valuable
alternative for mitigating the influence of outliers, and consequently,
reducing deviations from the underlying parametric models. Incorporating
robust statistical methods into extreme value theory has been shown to enhance
both the accuracy and reliability of parameter estimates \citep[][]{DE2006}.
\cite{JS2004} appear to be the first to employ the Minimum Density Power
Divergence (MDPD) introduced by \cite{Basu98} for the robust estimation of
parameters in extreme value distributions. Since then, this divergence measure
has become one of the most widely used tools for robust parameter estimation
in the context of extreme value theory. Several authors, including
\cite{Kim2008}, \cite{DGG13}, \cite{GGV2014}, \cite{DGG2021} have employed the
MDPD framework to estimate the tail index and quantiles of Pareto-type
distributions.\medskip\ 

\noindent Recently, \cite{Ghosh2017} proposed a robust MDPD-based estimator
for a real-valued tail index, serving as a robust generalization of the
estimator introduced by \cite{MB2003}. This approach also accounts for the
non-identically distributed structure of the exponential regression model,
following the methodology of \cite{GB2013}. Additionally, \cite{DGG13} applied
the MDPD framework to an extended Pareto distribution to model relative
excesses over a high threshold. More recently, \cite{MWG2023} developed a
robust estimator for the tail index of a Pareto-type distribution by applying
the MDPD approach within the exponential regression model framework.\medskip

\noindent We are thus faced with two main challenges: robustness and the
instability of the optimal sample fraction used in the computation of tail
index estimators. To address these issues, we propose assigning a weight
function $J$ to the density power divergence, leading to a robust
generalization of the weighted least squares tail index estimator
$\widehat{\gamma}_{k,J}.$ To the best of our knowledge, this approach has not
yet been explored in the existing literature. The proposed estimation
procedure is presented in detail in Section $\ref{sec3}.$\medskip

\noindent The remainder of this paper is organized as follows. Section
\ref{sec2} provides a brief overview of the MDPD estimation method, originally
introduced by \cite{Basu98}. The proposed weighted MDPD-based estimation of
the tail index is developed in Section \ref{sec3}. Section \ref{sec4} presents
the asymptotic properties of the proposed estimator, specifically its
consistency and asymptotic normality, with detailed proofs deferred to Section
\ref{sec6}. The finite-sample performance of the estimator is evaluated in
Section \ref{sec5} through a simulation study, including comparisons with
existing estimators. Section \ref{sec7} collects several technical lemmas and
propositions, while Section \ref{sec8} gathers the figures corresponding to
the simulation results.

\newpage

\section{\textbf{Minimum density power divergence\label{sec2}}}

\noindent Given two probability densities $\ell$ and $h,$ \cite{Basu98}
introduced a new distance between them called the density power divergence
\begin{equation}
d_{\alpha}\left(  h,\ell\right)  =\left\{
\begin{array}
[c]{ll}%
\int_{\mathbb{R}}\left[  \ell^{1+\alpha}\left(  x\right)  -\left(  1+\dfrac
{1}{\alpha}\right)  \ell^{\alpha}\left(  x\right)  h\left(  x\right)
+\dfrac{1}{\alpha}h^{1+\alpha}\left(  x\right)  \right]  dx, & \alpha
>0\medskip\\
\int_{\mathbb{R}}h\left(  x\right)  \log\dfrac{h\left(  x\right)  }%
{\ell\left(  x\right)  }dx, & \alpha=0,
\end{array}
\right.  \label{d}%
\end{equation}
where $\alpha$ is a nonnegative tuning parameter. The case corresponding to
$\alpha=0$ is obtained from the general case by letting $\alpha\rightarrow0$
leading to the classical Kullback-Leibler divergence denoted $d_{0}\left(
h,\ell\right)  .$ Let us consider a parametric model of densities $\left\{
\ell_{\theta}:\Theta\subset\mathbb{R}^{p}\right\}  $ and suppose that we
consider the estimation of the parameter $\theta.$ Let $H$ be the cdf
corresponding to the density $h.$ The minimum density power divergence (MDPD)
is a functional $T_{\alpha}\left(  H\right)  $ defined by $d_{\alpha}\left(
h,\ell_{T_{\alpha}\left(  H\right)  }\right)  =\min_{\theta\in\Theta}%
d_{\alpha}\left(  h,\ell_{\theta}\right)  .$ It is clear that the term $\int
h^{1+\alpha}\left(  x\right)  dx$ in $\left(  \ref{d}\right)  $ does not
contribute in the minimization of $d_{\alpha}\left(  h,\ell_{\theta}\right)  $
over $\theta\in\Theta.$ Then minimization in the computation of the MDPD
functional $T_{\alpha}\left(  H\right)  $ reduces to%
\begin{equation}
\delta_{\alpha}\left(  h;\theta\right)  :=\left\{
\begin{array}
[c]{lc}%
\int_{\mathbb{R}}\ell_{\theta}^{1+\alpha}\left(  x\right)  dx-\left(
1+\dfrac{1}{\alpha}\right)  \int_{\mathbb{R}}\ell_{\theta}^{\alpha}\left(
x\right)  dH\left(  x\right)  , & \alpha>0,\medskip\\
-\int_{\mathbb{R}}\log\ell_{\theta}\left(  x\right)  dH\left(  x\right)  , &
\alpha=0.
\end{array}
\right.  \label{dstart}%
\end{equation}
Given a random sample $Z_{1},...,Z_{n}$ from the distribution $H$ we may
estimate the objective function $h$ in $\left(  \ref{dstart}\right)  $ by
substituting $H$ with its empirical counterpart $H_{n}.$ For a given tuning
parameter $\alpha,$ the MDPD estimator $\widehat{\theta}_{n,\alpha}$ of
$\theta$ may be obtained by minimizing (over $\theta\in\Theta)$ the quantity%
\[
\delta_{n.\alpha}^{\ast}\left(  \theta\right)  :=\left\{
\begin{array}
[c]{lc}%
\int_{\mathbb{R}}\ell_{\theta}^{1+\alpha}\left(  x\right)  dx-\left(
1+\dfrac{1}{\alpha}\right)  \int_{\mathbb{R}}\ell_{\theta}^{\alpha}\left(
x\right)  dH_{n}\left(  x\right)  , & \alpha>0,\medskip\\
-\int_{\mathbb{R}}\log\ell_{\theta}\left(  x\right)  dH_{n}\left(  x\right)
, & \alpha=0,
\end{array}
\right.  ,
\]
which in turn equals
\[
\left\{
\begin{array}
[c]{lc}%
\int_{\mathbb{R}}\ell_{\theta}^{1+\alpha}\left(  x\right)  dx-\left(
1+\dfrac{1}{\alpha}\right)  \dfrac{1}{n}%
{\displaystyle\sum\limits_{i=1}^{n}}
\ell_{\theta}^{\alpha}\left(  Z_{i}\right)  , & \alpha>0,\medskip\\
-\dfrac{1}{n}%
{\displaystyle\sum\limits_{i=1}^{n}}
\log\ell_{\theta}\left(  Z_{i}\right)  , & \alpha=0.
\end{array}
\right.  .
\]
Thus the MDPD estimator $\widehat{\theta}_{n,\alpha}$ of $\theta$ minimizing
$\delta_{n,\alpha}^{\ast}\left(  \theta\right)  $ will be a solution of the
following equation
\[
\left\{
\begin{array}
[c]{lc}%
\int_{\mathbb{R}}\dfrac{d}{d\theta}\ell_{\theta}^{1+\alpha}\left(  x\right)
dx-\left(  1+\dfrac{1}{\alpha}\right)  \dfrac{1}{n}%
{\displaystyle\sum\limits_{i=1}^{n}}
\dfrac{d}{d\theta}\ell_{\theta}^{\alpha}\left(  Z_{i}\right)  =0, &
\alpha>0,\medskip\\
\dfrac{1}{n}%
{\displaystyle\sum\limits_{i=1}^{n}}
\dfrac{d}{d\theta}\log\ell_{\theta}\left(  Z_{i}\right)  =0, & \alpha=0.
\end{array}
\right.
\]
The tuning parameter $\alpha$ plays a crucial role in the MDPD framework, as
it governs the trade-off between efficiency and robustness. Specifically, when
$\alpha$ is close to zero, the estimator is more efficient but less robust to
outliers. Conversely, as $\alpha$ increases, robustness improves at the cost
of some loss in efficiency. It has been shown that estimators with small
values of $\alpha$ can exhibit strong robustness properties while incurring
only a slight loss in asymptotic efficiency relative to the maximum likelihood
estimator under correct model specification.

\section{\textbf{Weighted MDPD estimation of the tail index\label{sec3}}}

\noindent As stated in Lemma $\ref{lemma1},$ we have%
\begin{equation}
\gamma_{u}\left(  F\right)  :=\beta^{-1}\int_{1}^{\infty}\log xdF_{u,J}\left(
x\right)  =\gamma+o\left(  1\right)  , \label{lim1}%
\end{equation}
as $u\rightarrow\infty,$ where%
\[
F_{u,J}\left(  x\right)  :=\int_{1}^{x}J\left(  \overline{F}_{u}\left(
t\right)  \right)  dF_{u}\left(  t\right)  ,\text{ with }F_{u}\left(
t\right)  =F\left(  ut\right)  /\overline{F}\left(  t\right)  ,
\]
and $J$ is the weight function introduced in $\left(  \ref{HLM}\right)  .$
From equation $\left(  \ref{lim1}\right)  ,$ we deduce that $\gamma_{u}\left(
F\right)  ,$ for large $u,$ represents the tail functional associated with the
tail index estimator $\widehat{\gamma}_{k,J}$ defined in $\left(
\ref{HLM}\right)  .$ Since $J\left(  t\right)  $ vanishes at $t=1$ and
$\overline{F}$ is continuous, it is more appropriate to use $J\left(
\overline{F}\left(  ux\right)  /\overline{F}\left(  u-\right)  \right)  $
instead of $J\left(  \overline{F}\left(  ux\right)  /\overline{F}\left(
u\right)  \right)  ,$ where
\[
\overline{F}\left(  t-\right)  :=\lim_{\epsilon\uparrow0}\overline{F}\left(
t\right)  .
\]
Accordingly, the integral in $\left(  \ref{lim1}\right)  $ becomes
\[
\int_{1}^{\infty}\log xdF_{u,J}\left(  x\right)  =\int_{u}^{\infty}J\left(
\dfrac{\overline{F}\left(  ux\right)  }{\overline{F}\left(  u-\right)
}\right)  \log\dfrac{x}{u}d\dfrac{F\left(  ux\right)  }{\overline{F}\left(
u\right)  }.
\]
This modification will be adopted throughout the empirical estimation of tail
functionals. Specifically, in $\gamma_{u}\left(  F\right)  ,$ we replace the
threshold $u$ with the intermediate order statistic $X_{n-k:n}$ and the cdf
$F$ with its empirical counterpart
\[
F_{n}\left(  x\right)  :=n^{-1}\sum_{i=1}^{n}\mathbb{I}_{\left\{  X_{i:n}\leq
x\right\}  },
\]
to obtain%
\begin{align*}
&  \beta^{-1}\int_{X_{n-k:n}}^{\infty}J\left(  \dfrac{\overline{F}_{n}\left(
x\right)  }{\overline{F}_{n}\left(  X_{n-k:n}-\right)  }\right)  \log\dfrac
{x}{X_{n-k:n}}d\dfrac{F_{n}\left(  x\right)  }{\overline{F}_{n}\left(
X_{n-k:n}\right)  }\\
&  =\dfrac{\beta^{-1}}{k}%
{\displaystyle\sum\limits_{i=1}^{k}}
J\left(  \dfrac{i}{k+1}\right)  \log\dfrac{X_{n-i+1:n}}{X_{n-k:n}},
\end{align*}
which corresponds exactly to the expression of $\widehat{\gamma}_{k,J}.$ On
the other hand, under appropriate assumptions on the weight function $J,$ we
stated in Lemma $\ref{lemma1}$ that assertion $\left(  \ref{lim1}\right)  $ is
equivalent to%
\begin{equation}
\int_{1}^{\infty}f_{u,J}\left(  x\right)  \dfrac{d}{d\gamma}\log\ell
_{\gamma,J}\left(  x\right)  dx=o\left(  1\right)  , \label{lim3}%
\end{equation}
as $u\rightarrow\infty,$ or equivalently
\begin{equation}
\int_{1}^{\infty}J\left(  \dfrac{\overline{F}\left(  ux\right)  }{\overline
{F}\left(  u\right)  }\right)  \log\ell_{\gamma,J}\left(  x\right)
d\dfrac{F\left(  ux\right)  }{\overline{F}\left(  u\right)  }=o\left(
1\right)  , \label{lim3-prime}%
\end{equation}
as $u\rightarrow\infty,$ where $f_{u}\left(  x\right)  :=dF_{u}\left(
x\right)  /dx,$
\begin{equation}
\ell_{\gamma,J}\left(  x\right)  :=J\left(  x^{-1/\gamma}\right)  \ell
_{\gamma}\left(  x\right)  \label{l-J-gamma}%
\end{equation}
and
\begin{equation}
\ell_{\gamma}\left(  x\right)  :=\gamma^{-1}x^{-1-1/\gamma},\text{ }x>1,\text{
}\gamma>0, \label{l-gamma}%
\end{equation}
denotes the standard Pareto density function. From this, we infer that the
integral equation above leads to the same tail functional $\gamma_{u}\left(
F\right)  .$ In other words, the estimator $\widehat{\gamma}_{k,J}$ may be
viewed as the solution to the empirical counterpart of that equation, or
equivalently $\left(  \ref{lim3-prime}\right)  ,$ which reads:%
\[
\int_{1}^{\infty}J\left(  \dfrac{\overline{F}_{n}\left(  X_{n-k:n}x\right)
}{\overline{F}_{n}\left(  X_{n-k:n}-\right)  }\right)  \log\ell_{\gamma
,J}\left(  x\right)  d\dfrac{F_{n}\left(  X_{n-k:n}x\right)  }{\overline
{F}_{n}\left(  X_{n-k:n}\right)  }=0.
\]
In conclusion, the tail index estimator $\widehat{\gamma}_{k,J}$ can be viewed
as the result of minimizing the Kullback--Leibler divergence between two
weighted densities, $f_{u,J}$ and $\ell_{\gamma,J}.$ This leads us to consider
the general case $\alpha>0$ of the weighted minimum density power divergence
(WMDPD), between $f_{u}$ and $\ell_{\gamma},$ by
\[
d_{\alpha,J}\left(  f_{u},\ell_{\gamma}\right)  :=d_{\alpha}\left(
f_{u,J},\ell_{\gamma,J}\right)  ,
\]
in the sense that
\begin{equation}
d_{\alpha,J}\left(  f_{u},\ell_{\gamma}\right)  :=\int_{1}^{\infty}\left(
\ell_{\gamma,J}^{1+\alpha}\left(  x\right)  -\left(  1+\dfrac{1}{\alpha
}\right)  \ell_{\gamma,J}^{\alpha}\left(  x\right)  f_{u,J}\left(  x\right)
+\dfrac{1}{\alpha}f_{u,J}^{1+\alpha}\left(  x\right)  \right)  dx.
\label{WMDPD}%
\end{equation}
Since the term $\int f_{u,J}^{1+\alpha}\left(  x\right)  dx$ in $\left(
\ref{WMDPD}\right)  $ plays no role in the minimization of $d_{\alpha
,J}\left(  f_{u},\ell_{\gamma}\right)  $ over $\gamma>0,$ it is sufficient to
minimize the simplified divergence functional
\begin{equation}
d_{\alpha,J}^{\ast}\left(  f_{u},\ell_{\gamma}\right)  :=\int_{1}^{\infty}%
\ell_{\gamma,J}^{1+\alpha}\left(  x\right)  dx-\left(  1+\dfrac{1}{\alpha
}\right)  \int_{1}^{\infty}\ell_{\gamma.J}^{\alpha}\left(  x\right)
dF_{u,J}\left(  x\right)  , \label{deltau}%
\end{equation}
for sufficiently large $u.$ In the other terms, the WMDPD $\gamma_{u,\alpha
,J}$ of $\gamma$ minimizing $d_{u,\alpha,J}^{\ast}\left(  \gamma\right)  $ is
solution of the following equation%
\[
\int_{1}^{\infty}\dfrac{d}{d\gamma}\ell_{\gamma,J}^{\alpha}\left(  x\right)
J\left(  \dfrac{\overline{F}\left(  ux\right)  }{\overline{F}\left(
u-\right)  }\right)  d\dfrac{F\left(  ux\right)  }{\overline{F}\left(
u\right)  }=\dfrac{\alpha}{\alpha+1}\int_{1}^{\infty}\dfrac{d}{d\gamma}%
\ell_{\gamma,J}^{1+\alpha}\left(  x\right)  dx,
\]
By substituting $F$ with $F_{n}$ and $u$ with $X_{n-k:n},$ in the previous
equation, we end up with the WMDPD estimator $\widehat{\gamma}_{k,\alpha,J}$
of tail index $\gamma$ as solution of the equation
\begin{equation}
\dfrac{1}{k}%
{\displaystyle\sum\limits_{i=1}^{k}}
J\left(  \dfrac{i}{k+1}\right)  \dfrac{d}{d\gamma}\ell_{\gamma,J}^{\alpha
}\left(  \dfrac{X_{n-i+1:n}}{X_{n-k:n}}\right)  =\dfrac{\alpha}{\alpha+1}%
\int_{1}^{\infty}\dfrac{d}{d\gamma}\ell_{\gamma,J}^{1+\alpha}\left(  x\right)
dx,\text{ for }\alpha>0. \label{L}%
\end{equation}
We may therefore regard $\widehat{\gamma}_{k,\alpha,J}$ for $\alpha>0$ as a
robust generalization of $\widehat{\gamma}_{k,J}.$ It is also worth noting
that $\widehat{\gamma}_{k,\alpha,J}$ can be seen as a weighted extension of
the MDPD estimator $\widehat{\gamma}_{k,\alpha}:=\widehat{\gamma}%
_{k,\alpha,J_{0}}$ previously introduced by \cite{DGG13}. \smallskip

\begin{center}%
\begin{table}[tbp] \centering
\begin{tabular}
[t]{|c|c|c|c|c|c|}\hline
& $J_{\log}$ & $J_{0}$ & $J_{2}$ & $J_{3}$ & $J_{4}$\\\hline
$\mathcal{L}\left(  s\right)  $ & $0$ & $0$ & $4s^{2}\dfrac{\log s}{s^{2}-1}$
& $6s^{2}\dfrac{\log s}{s^{2}-1}$ & $8s^{2}\dfrac{\log s}{s^{2}-1}$\\\hline
$\mathcal{L}^{\left(  1\right)  }\left(  s\right)  $ & $0$ & $0$ &
$4s\dfrac{s^{2}-2\log s-1}{\left(  s^{2}-1\right)  ^{2}}$ & $6s\dfrac
{s^{2}-2\log s-1}{\left(  s^{2}-1\right)  ^{2}}$ & $8s\dfrac{s^{2}-2\log
s-1}{\left(  s^{2}-1\right)  ^{2}}$\\\hline
\end{tabular}%
\caption{}
\label{Tab1}%
\end{table}%

\end{center}

\section{\textbf{Main results\label{sec4}}}

\noindent Next, we establish the existence and consistency of a sequence of
solutions to the estimating equation $\left(  \ref{L}\right)  ,$ for
$\alpha>0.$ To this end, we consider a class of weight functions that satisfy
the following regularity conditions:

\begin{itemize}
\item $\left[  A1\right]  $ $J$ is continuous nonincreasing nonnegative on the
interval $\left(  0,1\right)  $ with $\int_{0}^{1}J\left(  s\right)  ds=1.$

\item $\left[  A2\right]  $ $J$ and their $m$-th derivatives $\left\{
J^{\left(  m\right)  }\right\}  _{_{1\leq m\leq3}}$ are bounded on the
interval $\left(  0,1\right)  .$

\item $\left[  A3\right]  $ $\mathcal{L}\left(  s\right)  :=\dfrac{J^{\left(
1\right)  }\left(  s\right)  }{J\left(  s\right)  }\log s$ as well as their
$m$-th derivatives $\left\{  \mathcal{L}^{\left(  m\right)  }\right\}  _{1\leq
m\leq3}$ are bounded on the interval $\left(  0,1\right)  .$
\end{itemize}

\noindent The regularity assumptions $\left[  A1\right]  -\left[  A3\right]
$\ are relatively mild, as the commonly used weight functions listed $\left(
\ref{J}\right)  $\ satisfy them. In fact, Table $\ref{Tab1}$\ presents the
explicit expressions of the function $\mathcal{L}$\ and its first derivative
$\mathcal{L}^{\left(  1\right)  }$\ for each of the aforementioned weight
functions. One can verify that both $\mathcal{L}$ and $\mathcal{L}^{\left(
1\right)  }$ are bounded over the interval $\left(  0,1\right)  $ and that
their second and third derivatives are also bounded. As we shall see, the
functions $\mathcal{L}$ and $\mathcal{L}^{\left(  m\right)  }$ play a central
role in the proofs of $\ref{lemma1}$ and Theorems $\ref{Theorem1}$ and
$\ref{Theorem2}.$\smallskip

\noindent From this point onward, we denote the true value of the tail index
$\gamma$ by $\gamma_{0}.$

\begin{theorem}
\textbf{\label{Theorem1}}(existence and consistency) Assume that tail
distribution $\overline{F}$ satisfies the first-order condition $\left(
\ref{RV}\right)  $ and let $J$ be a continuous function fulfilling assumptions
$\left[  A1\right]  -\left[  A3\right]  .$ For a given sequence of integers
$k=k_{n}$ such that $k\rightarrow\infty$ and $k/n\rightarrow0,$ as
$n\rightarrow\infty,$ with probability tending to $1,$ there exists solution
$\widehat{\gamma}_{k,\alpha,J},$ for the estimating equation $\left(
\ref{L}\right)  $ such that $\widehat{\gamma}_{k,\alpha,J}\overset{\mathbf{P}%
}{\rightarrow}\gamma_{0},$ as $n\rightarrow\infty,$ provided that $\alpha>0.$
\end{theorem}

\noindent Since weak approximations of extreme value theory based statistics
are achieved in the second-order framework \citep[see e.g.][]{deHS96}, then it
seems quite natural to suppose that $\overline{F}$ satisfies the well-known
second-order condition of regular variation. That is, we assume that for any
$x>0$%
\begin{equation}
\lim_{t\rightarrow\infty}\frac{\overline{F}\left(  tx\right)  /\overline
{F}\left(  t\right)  -x^{-1/\gamma_{0}}}{A\left(  t\right)  }=x^{-1/\gamma
_{0}}\dfrac{x^{\tau/\gamma_{0}}-1}{\tau\gamma_{0}}, \label{second-order}%
\end{equation}
where $\tau<0$ denotes the second-order parameter, $A$\ is a function tending
to zero and not changing signs near infinity with regularly varying absolute
value with index $\tau/\gamma_{0}.$ For convenience, we set $a\left(
t\right)  :=A\left(  U\left(  t\right)  \right)  $\textbf{\ }and $U\left(
t\right)  :=F^{\leftarrow}\left(  1-1/t\right)  ,$ $t>1,$ where $F^{\leftarrow
}\left(  s\right)  :=\inf\left\{  x:F\left(  x\right)  \geq s\right\}  ,$ for
$0<s<1,$ stands for the quantile function pertaining to cdf $F.$

\begin{theorem}
\textbf{\label{Theorem2}}(asymptotic normality) Assume that tail distribution
$\overline{F}$ satisfies the second-order condition $\left(
\ref{second-order}\right)  $ and that, for $\alpha>0,$ the estimator
$\widehat{\gamma}_{k,\alpha,J}$ is a consistent for $\gamma_{0}$ satisfying
$\left(  \ref{L}\right)  .$ Let $J$ be a continuous function fulfilling
assumptions $\left[  A1\right]  -\left[  A3\right]  $ and $k$ be a sequence of
integers such that $k\rightarrow\infty,$ $k/n\rightarrow0$ and $\sqrt
{k}a\left(  n/k\right)  $ is asymptotically bounded. Then, in the probability
space $\left(  \Omega,\mathcal{A},\mathbf{P}\right)  $ there exists a standard
Wiener process $\left\{  W\left(  x\right)  ,\text{ }x\geq0\right\}  ,$ such
that
\begin{align*}
&  \left(  1+\frac{1}{\alpha}\right)  \eta_{\gamma_{0},\alpha}\sqrt{k}\left(
\widehat{\gamma}_{k,\alpha,J}-\gamma_{0}\right) \\
&  =\int_{1}^{\infty}\left(  W\left(  x^{-1/\gamma_{0}}\right)  -x^{-1/\gamma
_{0}}W\left(  1\right)  \right)  J\left(  x^{-1/\gamma_{0}}\right)
d\Psi_{\gamma_{0},\alpha}^{\left(  1\right)  }+\sqrt{k}a\left(  n/k\right)
B_{\gamma_{0},\alpha}^{\left(  1\right)  }+o_{\mathbf{P}}\left(  1\right)  ,
\end{align*}
where $\Psi_{\gamma_{0},\alpha}^{\left(  1\right)  }$ and $\eta_{\gamma
_{0},\alpha}$ are respectively given in $\left(  \ref{psi-m}\right)  $ and
$\left(  \ref{eta}\right)  ,$ and
\[
B_{\gamma_{0},\alpha}^{\left(  1\right)  }:=\int_{1}^{\infty}x^{-1/\gamma_{0}%
}\frac{x^{\tau/\gamma_{0}}-1}{\tau\gamma_{0}}J\left(  x^{-1/\gamma_{0}%
}\right)  d\Psi_{\gamma_{0},\alpha}^{\left(  1\right)  }\left(  x^{-1/\gamma
_{0}}\right)  ,
\]
Whenever $\sqrt{k}a\left(  n/k\right)  \rightarrow\lambda<\infty,$
\[
\left(  1+\frac{1}{\alpha}\right)  \eta_{\gamma_{0},\alpha}\sqrt{k}\left(
\widehat{\gamma}_{k,\alpha,J}-\gamma_{0}\right)  \overset{\mathcal{D}%
}{\rightarrow}\mathcal{N}\left(  \lambda B_{\gamma_{0},\alpha}^{\left(
1\right)  },\sigma_{\gamma_{0},\alpha}^{2}\right)  ,
\]
as $n\rightarrow\infty,$ where%
\[
\sigma_{\gamma_{0},\alpha}^{2}:=%
{\displaystyle\int_{0}^{1}}
{\displaystyle\int_{0}^{1}}
\left(  \min\left(  s,t\right)  -st\right)  J\left(  s\right)  J\left(
t\right)  d\Psi_{\gamma_{0},\alpha}^{\left(  1\right)  }\left(  s^{-\gamma
_{0}}\right)  d\Psi_{\gamma_{0},\alpha}^{\left(  1\right)  }\left(
t^{-\gamma_{0}}\right)  .
\]

\end{theorem}

\section{\textbf{Simulation study\label{sec5}}}

\noindent In this section, we address the following two points:

\begin{itemize}
\item The performance of WMDPD estimator $\widehat{\gamma}_{k,\alpha,J}%
$\ compared to that of non-weighted MDPD one $\widehat{\gamma}_{k,\alpha}$ \citep[][]{DGG13}.

\item Both the performance and robustness of MMDPD estimator $\widehat{\gamma
}_{k,\alpha,J}$ compared to those of WLSE one $\widehat{\gamma}_{k,J}$ \citep[][]{HLM2006}.
\end{itemize}

\noindent The comparison of these estimators is conducted under the following
two distributional models:

\begin{itemize}
\item Burr $\left(  \gamma,\delta\right)  $ distribution: defined by $F\left(
x\right)  =1-\left(  1+x^{1/\delta}\right)  ^{-\delta/\gamma},$ for $x>0,$
$\delta>0$ and $\gamma>0.$

\item Fr\'{e}chet $\left(  \gamma\right)  $ distribution: defined by $F\left(
x\right)  =\exp\left(  -x^{-1/\gamma}\right)  ,$ for $x>0$ and $\gamma
>0.\medskip$
\end{itemize}

\noindent To investigate the first issue, we use the triweight function
$J_{3}$ defined in $\left(  \ref{J}\right)  ,$ and set the Burr distribution
parameters to $\delta=1/4$ and $\gamma=0.5.$ We consider four values for the
robustness parameter $\alpha$, namely $\left\{  0,0.1,0.5,1\right\}  $ and
generate $2000$ samples of size $n=1000.$ Figures \ref{fig1} and \ref{fig2}
illustrate the behavior of the estimators $\widehat{\gamma}_{k,\alpha,J}$ and
$\widehat{\gamma}_{k,\alpha,1}$ as a function of the number $k.$ Further,
Figures \ref{fig3} and \ref{fig6} display the absolute bias (ABIAS) and mean
squared error (MSE) of $\widehat{\gamma}_{k,\alpha,J} $ and $\widehat{\gamma
}_{k,\alpha}$ for values $\gamma=0.15$ and $0.4,$ with $\alpha$ values of
$0.1,$ $0.5$ and $0.9.$ These results are based on $2000$ samples of size
$n=300$ and also presented as a function of $k.$ The simulation outcomes, as
depicted in Figures \ref{fig1} to \ref{fig6}, demonstrate that the estimator
$\widehat{\gamma}_{k,\alpha,J}$ consistently outperforms $\widehat{\gamma
}_{k,\alpha},$ particularly for larger values of $\alpha.$ While estimator
$\widehat{\gamma}_{k,\alpha}$ becomes increasingly sensitive as $\alpha,$ the
$\widehat{\gamma}_{k,\alpha,J}$ maintains greater stability, particularly as
the true value of $\gamma.$ Furthermore, both ABIAS and MSE increase as the
number of extreme values $k$ increases.\smallskip

\noindent We now turn to the second issue by considering an $\epsilon
$-contaminated model $F_{\epsilon}\left(  x\right)  :=\left(  1-\epsilon
\right)  F_{1}\left(  x\right)  +\epsilon F_{2}\left(  x\right)  $ as a
mixture of two cdf's $F_{1}$ and $F_{2},$ where $0<\epsilon<1$ denotes the
outliers contamination fraction. We examine four mixture scenarios from the
following models:\smallskip

\begin{itemize}
\item $\left[  S1\right]  $ Burr $\left(  \gamma,\delta\right)  $ contaminated
by Burr $\left(  2,0.5\right)  ,$

\item $\left[  S2\right]  $ Fr\'{e}chet $\left(  \gamma\right)  $ contaminated
by Fr\'{e}chet $\left(  2\right)  ,$

\item $\left[  S3\right]  $ Fr\'{e}chet $\left(  \gamma\right)  $ contaminated
by Burr $\left(  2,0.5\right)  ,$

\item $\left[  S4\right]  $ Burr $\left(  \gamma,\delta\right)  $ contaminated
by Fr\'{e}chet $\left(  2\right)  .\medskip$
\end{itemize}

\noindent We choose $\epsilon=0.1,$ set $\delta=1/4$, $\gamma=0.6$ and
consider three values for $\alpha$\ namely $\left\{  0.1,0.5,1\right\}  .$ For
each scenario, we generate $2000$ samples of size $500.$ We assess the
estimators $\widehat{\gamma}_{k,\alpha,J}$ and $\widehat{\gamma}_{k,J}$ in
terms of ABIAS and MSE under both contaminated and non-contaminated
distributions as a functions of the number $k.\medskip$

\noindent Regarding the four scenarios depicted in Figures \ref{fig7}%
-\ref{fig10}, we observe that introducing contamination affects both
estimators in all cases. However, the estimator $\widehat{\gamma}_{k,\alpha
,J}$ still outperforms the estimator $\widehat{\gamma}_{\alpha,J},$ as
indicated by the ABIAS and MSE values. Moreover, the performance of both
estimators improves as the number $k$ increases, even in the presence of contamination.

\section{\textbf{Proofs\label{sec6}}}

\noindent Let us begin to define the following expressions: for $m=1,2,3,4$%
\begin{equation}
\pi_{k}^{\left(  m\right)  }\left(  \gamma_{0}\right)  :=\int_{1}^{\infty}%
\Psi_{\gamma_{0},\alpha+1}^{\left(  m\right)  }\left(  x\right)  dx-\left(
1+\dfrac{1}{\alpha}\right)  A_{k}^{\left(  m\right)  }\left(  \gamma
_{0}\right)  , \label{pi}%
\end{equation}%
\begin{equation}
A_{k}^{\left(  m\right)  }\left(  \gamma_{0}\right)  :=\frac{1}{k}\sum
_{i=1}^{k}J\left(  \dfrac{i}{k+1}\right)  \Psi_{\gamma_{0},\alpha}^{\left(
m\right)  }\left(  \frac{X_{n-i+1:n}}{X_{n-k:n}}\right)  , \label{Am}%
\end{equation}
and
\begin{equation}
\Psi_{\gamma_{0},\alpha}^{\left(  m\right)  }\left(  x\right)  :=\left.
\frac{d^{m}\ell_{\gamma,J}^{\alpha}\left(  x\right)  }{d^{m}\gamma}\right\vert
_{\gamma=\gamma_{0}}, \label{psi-m}%
\end{equation}
denotes the values of $m$-th derivative of $\ell_{\gamma,J}^{\alpha},$ with
respect to $\gamma,$ at $\gamma_{0}.$

\subsection{Proof of Theorem $\ref{Theorem1}$}

To show the existence and consistency of $\widehat{\gamma}_{k,\alpha,J}$ we
will adapt the proof of Theorem 1in \cite{DGG13} which itself is an adaptation
of the proof of Theorem 5.1 in Chapter 6 of \cite{LC98}. That theorem
establishes the existence and consistency of solutions to the likelihood
equations, and we extend it here to the WMDPDE framework. Let $d_{n,\alpha
,J}^{\ast}\left(  \gamma\right)  $ denotes the empirical counterpart of the
weighted density power divergence objective function $d_{u,\alpha,J}^{\ast
}\left(  \gamma\right)  ,$ given in $\left(  \ref{deltau}\right)  ,$ namely
\[
\widehat{d}_{n,\alpha,J}^{\ast}\left(  \gamma\right)  :=\int_{1}^{\infty}%
\ell_{\gamma,J}^{1+\alpha}\left(  x\right)  dx-\left(  1+\dfrac{1}{\alpha
}\right)  \frac{1}{k}\sum_{i=1}^{k}J\left(  \dfrac{i}{k+1}\right)
\ell_{\gamma,J}^{\alpha}\left(  \frac{X_{n-i+1:n}}{X_{n-k:n}}\right)  ,\text{
}\alpha>0.
\]
Next we show that
\begin{equation}
\mathbf{P}_{\gamma_{0}}\left(  d_{n,\alpha,J}^{\ast}\left(  \gamma_{0}\right)
<d_{n,\alpha,J}^{\ast}\left(  \gamma\right)  ,\text{ for all }\gamma\in
I_{\epsilon}\right)  \rightarrow1,\text{ as }\epsilon\downarrow0, \label{P}%
\end{equation}
where $I_{\epsilon}:=\left(  \gamma_{0}-\epsilon,\gamma_{0}+\epsilon\right)
,$ for $0<\epsilon<\gamma_{0}.$ By applying Taylor's expansion near
$\gamma_{0}$ to function $\gamma\rightarrow d_{n,\alpha,J}^{\ast}\left(
\gamma\right)  ,$ we decompose $d_{n,\alpha,J}^{\ast}\left(  \gamma\right)
-d_{n,\alpha,J}^{\ast}\left(  \gamma_{0}\right)  $ into%
\[%
\begin{array}
[c]{ll}
& \pi_{k}^{\left(  1\right)  }\left(  \gamma_{0}\right)  \left(  \gamma
-\gamma_{0}\right)  +2^{-1}\pi_{k}^{\left(  2\right)  }\left(  \gamma
_{0}\right)  \left(  \gamma-\gamma_{0}\right)  ^{2}+6^{-1}\pi_{k}^{\left(
3\right)  }\left(  \widetilde{\gamma}\right)  \left(  \gamma-\gamma
_{0}\right)  ^{3}\medskip\\
=: & S_{1,k}+S_{2,k}+S_{3,k},
\end{array}
\]
where $\widetilde{\gamma}$ is between $\gamma$ and $\gamma_{0}.$ In the first
assertion of Lemma $\ref{lemma2}$ we stated that $\pi_{k}^{\left(  1\right)
}\left(  \gamma_{0}\right)  \overset{\mathbf{P}}{\rightarrow}0,$ as
$n\rightarrow\infty.$ In other terms, for any $\epsilon>0$ sufficiently small
we have $\left\vert \pi_{k}^{\left(  1\right)  }\left(  \gamma_{0}\right)
\right\vert <\epsilon^{2},$ which entails that $\left\vert S_{1,k}\right\vert
<\epsilon^{3},$ for ever $\gamma\in I_{\epsilon},$ with probability tending to
$1.$ From the second assertion, we have $\pi_{k}^{\left(  2\right)  }\left(
\gamma_{0}\right)  \overset{\mathbf{P}}{\rightarrow}\eta_{\gamma_{0},\alpha}$
as $n\rightarrow\infty,$ where $\eta_{\gamma_{0},\alpha}$ is that given in
$\left(  \ref{eta}\right)  ,$ therefore $S_{2,k}=\left(  1+o_{\mathbf{P}%
}\left(  1\right)  \right)  2^{-1}\eta_{\gamma_{0},\alpha}\left(
\gamma-\gamma_{0}\right)  ^{2}.$ The constant $\eta_{\gamma_{0},\alpha}$ being
positive, then there exists $c>0$ and $\epsilon_{0}>0$ such that for
$\epsilon<\epsilon_{0}$ such that $S_{2,k}>c\epsilon^{2}.$ Note that
$\widetilde{\gamma}$ is a consistent estimator for $\gamma_{0},$ then by Lemma
$\ref{lemma4}$ we deduce that $\pi_{k}^{\left(  3\right)  }\left(
\widetilde{\gamma}\right)  =O_{\mathbf{P}}\left(  1\right)  .$ This means that
with probability tending to $1$ there exists a constant $d>0$ such that
$\left\vert S_{3,k}\right\vert <d\epsilon^{3}.$ Therefore
\[
\min_{\gamma\in I_{\epsilon}}\left(  S_{1,k}+S_{2,k}+S_{3,k}\right)
>c\epsilon^{2}-\left(  d+1\right)  \epsilon^{3},
\]
with probability tending to $1.$ The choice $0<\epsilon<c/\left(  d+1\right)
$ implies that $c\epsilon^{2}-\left(  d+1\right)  \epsilon^{3}>0$ which leads
to inequality $\left(  \ref{P}\right)  .$ To complete the proof of the
existence and consistency, we follow the same steps to those used in the proof
of Theorem 3.7 in Chapter 6 of \cite{LC98}. Let $\epsilon>0$ be small so that
$0<\epsilon<c/\left(  d+1\right)  $ and $I_{\epsilon}\subset\left(
0,\infty\right)  ,$ then consider the set
\[
S_{n}\left(  \epsilon\right)  :=\left\{  \gamma:d_{n,\alpha,J}^{\ast}\left(
\gamma_{0}\right)  <d_{n,\alpha,J}^{\ast}\left(  \gamma\right)  \text{ for all
}\gamma\in I_{\epsilon}\right\}  .
\]
We already showed that $\mathbf{P}_{\gamma_{0}}\left\{  S_{n}\left(
\epsilon\right)  \right\}  \rightarrow1$ for any such $\epsilon,$ then there
exists a sequence $\epsilon_{n}\downarrow0$ such that $\mathbf{P}_{\gamma_{0}%
}\left\{  S_{n}\left(  \epsilon_{n}\right)  \right\}  \rightarrow1$ as
$n\rightarrow\infty.$ Note that $\gamma\rightarrow d_{n,\alpha,J}^{\ast
}\left(  \gamma\right)  $ being differentiable on $\left(  0,\infty\right)  ,$
then given $\gamma\in S_{n}\left(  \epsilon_{n}\right)  $ there exists a point
$\widehat{\gamma}_{k,\alpha,J}\left(  \epsilon_{n}\right)  \in I_{\epsilon
_{n}}$ for which $d_{n,\alpha,J}^{\ast}\left(  \gamma\right)  $ attains a
local minimum, thereby $\pi_{k}^{\left(  1\right)  }\left(  \widehat{\gamma
}_{k,\alpha,J}\left(  \epsilon_{n}\right)  \right)  =0.$ Let us set
$\widehat{\gamma}_{k,\alpha,J}^{\ast}=\widehat{\gamma}_{k,\alpha,J}\left(
\epsilon_{n}\right)  $ for $\gamma\in S_{n}\left(  \epsilon_{n}\right)  $ and
arbitrary otherwise. Obviously
\[
\mathbf{P}_{\gamma_{0}}\left\{  \pi_{k}^{\left(  1\right)  }\left(
\widehat{\gamma}_{k,\alpha,J}^{\ast}\right)  =0\right\}  \geq\mathbf{P}%
_{\gamma_{0}}\left\{  S_{n}\left(  \epsilon_{n}\right)  \right\}
\rightarrow1,\text{ as }n\rightarrow\infty,
\]
thus with probability tending to $1$ there exists a sequence of solutions to
estimating equation $\left(  \ref{L}\right)  .$ Observe now that for any fixed
$\epsilon>0$ and $n$ sufficiently large $\mathbf{P}_{\gamma_{0}}\left\{
\left\vert \widehat{\gamma}_{k,\alpha,J}^{\ast}-\gamma_{0}\right\vert
<\epsilon\right\}  \geq\mathbf{P}_{\gamma_{0}}\left\{  \left\vert
\widehat{\gamma}_{k,\alpha,J}^{\ast}-\gamma_{0}\right\vert <\epsilon
_{n}\right\}  \rightarrow1,$ as $n\rightarrow\infty,$ which establishes the
consistency of $\widehat{\gamma}_{k,\alpha,J}^{\ast},$ as sought.

\subsection{Proof of Theorem $\ref{Theorem2}$}

Let us retain the notations used in the proof of Theorem $\ref{Theorem1}$,
namely those of $\Psi_{\gamma_{0},\alpha}^{\left(  1\right)  }$ and $\pi
_{k}^{\left(  m\right)  }\left(  \gamma_{0}\right)  .$ Applying Taylor's
expansion to the estimating equation $\pi_{k}^{\left(  1\right)  }\left(
\widehat{\gamma}_{k,\alpha,J}\right)  =0,$ yields%
\[
0=\pi_{k}^{\left(  1\right)  }\left(  \gamma_{0}\right)  +\pi_{k}^{\left(
2\right)  }\left(  \gamma_{0}\right)  \left(  \widehat{\gamma}_{k,\alpha
,J}-\gamma_{0}\right)  +\frac{1}{2}\pi_{k}^{\left(  3\right)  }\left(
\widehat{\gamma}_{0}\right)  \left(  \widehat{\gamma}_{k,\alpha,J}-\gamma
_{0}\right)  ^{2},
\]
where $\widetilde{\gamma}_{0}$ is between $\gamma_{0}$ and $\widehat{\gamma
}_{k,\alpha,J}.$ We showed above that $\widehat{\gamma}_{k,\alpha
,J}\overset{\mathbf{P}}{\rightarrow}\gamma_{0}$ it follows that
$\widetilde{\gamma}_{0}\overset{\mathbf{P}}{\rightarrow}\gamma_{0}$ as
$n\rightarrow\infty,$ thus in view of Lemma $\ref{lemma4}$ we get $\pi
_{k}^{\left(  3\right)  }\left(  \widetilde{\gamma}_{0}\right)  =O_{\mathbf{P}%
}\left(  1\right)  .$ This implies that $2^{-1}\pi_{k}^{\left(  3\right)
}\left(  \widehat{\gamma}_{0}\right)  \left(  \widehat{\gamma}_{k,\alpha
,J}-\gamma_{0}\right)  ^{2}=o_{\mathbf{P}}\left(  1\right)  \left(
\widehat{\gamma}_{k,\alpha,J}-\gamma_{0}\right)  ,$ thereby
\[
\pi_{k}^{\left(  2\right)  }\left(  \gamma_{0}\right)  \sqrt{k}\left(
\widehat{\gamma}_{k,\alpha,J}-\gamma_{0}\right)  \left(  1+o_{\mathbf{P}%
}\left(  1\right)  \right)  =-\sqrt{k}\pi_{k}^{\left(  1\right)  }\left(
\gamma_{0}\right)  ,
\]
as $n\rightarrow\infty.$ Once again making use of Gaussian approximation
$\left(  \ref{app1}\right)  $ of Lemma $\ref{lemma3},$ yields
\begin{align*}
&  \left(  1+\frac{1}{\alpha}\right)  \pi_{k}^{\left(  2\right)  }\left(
\gamma_{0}\right)  \sqrt{k}\left(  \widehat{\gamma}_{k,\alpha,J}-\gamma
_{0}\right) \\
&  =\int_{1}^{\infty}\left(  W\left(  x^{-1/\gamma_{0}}\right)  -x^{-1/\gamma
_{0}}W\left(  1\right)  \right)  J\left(  x^{-1/\gamma_{0}}\right)
d\Psi_{\gamma_{0},\alpha}^{\left(  1\right)  }\left(  x\right)  +\lambda
B_{\gamma_{0},\alpha}^{\left(  1\right)  }+o_{\mathbf{P}}\left(  1\right)  ,
\end{align*}
where $\left\{  W\left(  x\right)  ,\text{ }x\geq0\right\}  $ is a standard
Wiener process and $B_{\gamma_{0},\alpha}^{\left(  1\right)  }$ being the
asymptotic bias given in Theorem $\ref{Theorem2}.$ From the second assertion
of Lemma $\ref{lemma2},$ we have $\pi_{k}^{\left(  2\right)  }\left(
\gamma_{0}\right)  \overset{\mathbf{P}}{\rightarrow}\eta_{\gamma_{0},\alpha}$
therefore $\left(  1+\frac{1}{\alpha}\right)  \eta_{\gamma_{0},\alpha}\sqrt
{k}\left(  \widehat{\gamma}_{k,\alpha,J}-\gamma_{0}\right)  \rightarrow
\mathcal{N}\left(  \lambda B_{\gamma_{0},\alpha}^{\left(  1\right)  }%
,\sigma_{\gamma_{0},\alpha}^{2}\right)  ,$ as $n\rightarrow\infty,$ where%
\[%
\begin{array}
[c]{l}%
\sigma_{\gamma_{0},\alpha}^{2}:=%
{\displaystyle\int_{1}^{\infty}}
{\displaystyle\int_{1}^{\infty}}
\left(  \min\left(  x^{-1/\gamma_{0}},y^{-1/\gamma_{0}}\right)  -x^{-1/\gamma
_{0}}y^{-1/\gamma_{0}}\right)  \medskip\\
\ \ \ \ \ \ \ \ \ \ \ \ \ \ \ \ \ \ \ \ \ \ \ \ \times J\left(  x^{-1/\gamma
_{0}}\right)  J\left(  y^{-1/\gamma_{0}}\right)  d\Psi_{\gamma_{0},\alpha
}^{\left(  1\right)  }\left(  x\right)  d\Psi_{\gamma_{0},\alpha}^{\left(
1\right)  }\left(  y\right)  ,
\end{array}
\]
which equals $\int_{0}^{1}\int_{0}^{1}\left(  \min\left(  s,t\right)
-st\right)  J\left(  s\right)  J\left(  t\right)  d\Psi_{\gamma_{0},\alpha
}^{\left(  1\right)  }\left(  s^{-\gamma_{0}}\right)  d\Psi_{\gamma_{0}%
,\alpha}^{\left(  1\right)  }\left(  t^{-\gamma_{0}}\right)  ,$ this completes
the proof of Theorem $\ref{Theorem2}.\medskip$%
\[
\]%
\[
\]

\noindent\textbf{Conclusion.\medskip}

\noindent In this work, we proposed a new class of robust tail index
estimators for Pareto-type distributions by incorporating a weight function
into the density power divergence framework. This approach provides a flexible
and effective generalization of existing weighted least squares and
kernel-based estimators, enhancing robustness to outliers and model
deviations. We established the asymptotic properties of the proposed
estimators, including consistency and asymptotic normality, under general
regularity conditions. A comprehensive simulation study demonstrated the
improved finite-sample performance of our estimators compared to classical
methods, particularly in the presence of contamination or small departures
from the assumed model. These results highlight the relevance of robust
statistical techniques in extreme value theory and open promising directions
for future research, including extensions to multivariate settings and
dependent data structures.

\section{\textbf{Declarations}}

\begin{itemize}
\item The authors have no relevant financial or non-financial interests to disclose.

\item The authors have no competing interests to declare that are relevant to
the content of this article.

\item All authors certify that they have no affiliations with or involvement
in any organization or entity with any financial interest or non-financial
interest in the subject matter or materials discussed in this manuscript.

\item The authors have no financial or proprietary interests in any material
discussed in this article.
\end{itemize}

\section{\textbf{Appendix A\label{sec7}}}

\begin{lemma}
\textbf{\label{lemma1}}Assume that $\overline{F}$ satisfies the first-order
condition $\left(  \ref{RV}\right)  $ and let $J$ be a function fulfilling
assumptions $\left[  A1\right]  -\left[  A3\right]  ,$ then $\lim
_{u\rightarrow\infty}\beta^{-1}\int_{1}^{\infty}\log xdF_{u,J}\left(
x\right)  =\gamma,$ where $\beta:=-\int_{0}^{1}J\left(  s\right)  \log sds.$
Moreover the latter limit is equivalent that of
\[
\lim_{u\rightarrow\infty}\int_{1}^{\infty}f_{u,J}\left(  x\right)  \dfrac
{d}{d\gamma}\log\ell_{\gamma,J}\left(  x\right)  dx=0.
\]

\end{lemma}

\begin{proof}
Recalling the notations of Section $\ref{sec3},$ we write%
\[
\int_{1}^{\infty}\log xdF_{u,J}\left(  x\right)  =\int_{1}^{\infty}J\left(
\frac{\overline{F}\left(  ux\right)  }{\overline{F}\left(  u\right)  }\right)
\log xd\frac{F\left(  ux\right)  }{\overline{F}\left(  u\right)  }=:I_{u}.
\]
We have show that $I_{u}\rightarrow\gamma\beta,$ as $u\rightarrow\infty. $ To
this end, let us decompose $I_{u}$ into the sum of
\[
I_{u,1}:=-\int_{1}^{\infty}\left(  J\left(  \frac{\overline{F}\left(
ux\right)  }{\overline{F}\left(  u\right)  }\right)  -J\left(  x^{-1/\gamma
}\right)  \right)  \log xd\frac{\overline{F}\left(  ux\right)  }{\overline
{F}\left(  u\right)  },
\]%
\[
I_{u,2}:=-\int_{1}^{\infty}J\left(  x^{-1/\gamma}\right)  \log xd\left(
\frac{\overline{F}\left(  ux\right)  }{\overline{F}\left(  u\right)
}-x^{-1/\gamma}\right)
\]
and $I_{u,3}:=-\int_{1}^{\infty}J\left(  x^{-1/\gamma}\right)  \log
xdx^{-1/\gamma}.$ By using the mean value theorem, we write%
\[
I_{u,1}:=-\int_{1}^{\infty}\left(  \frac{\overline{F}\left(  ux\right)
}{\overline{F}\left(  u\right)  }-x^{-1/\gamma}\right)  J^{\left(  1\right)
}\left(  \theta_{u}\left(  x\right)  \right)  \log xd\frac{\overline{F}\left(
ux\right)  }{\overline{F}\left(  u\right)  },
\]
where $\theta_{u}\left(  x\right)  $ is between $x^{-1/\gamma}$ and
$\overline{F}\left(  ux\right)  /\overline{F}\left(  u\right)  $ and
$J^{\left(  1\right)  }$ beinh the first derivative of $J$. Next we use the
first-order Potter's inequality to the regularly varying function\textbf{\ }%
$\overline{F},$ that is: for every\textbf{\ }$0<\epsilon<1,$\textbf{\ }there
exists $u_{0}=u_{0}\left(  \epsilon\right)  ,$ such that for $u>u_{0}$ and
$x\geq1,$ $\left\vert \overline{F}\left(  ux\right)  /\overline{F}\left(
u\right)  -x^{-1/\gamma}\right\vert <\epsilon x^{\tau+\epsilon},$ see, for
instance, Proposition B.1.10 in \cite{deHF06} page 369. From assumption
$\left[  A2\right]  ,$ the function $J^{\left(  1\right)  }$ is bounded
function over $\left(  0,1\right)  ,$ then%
\[
I_{u,1}=o\left(  1\right)  \int_{1}^{\infty}x^{\tau+\epsilon}\log
xd\frac{\overline{F}\left(  ux\right)  }{\overline{F}\left(  u\right)  }.
\]
By using an integration by parts to the latter integral, yields%
\[
I_{u,1}=o\left(  1\right)  \int_{1}^{\infty}\frac{\overline{F}\left(
ux\right)  }{\overline{F}\left(  u\right)  }d\left(  x^{\tau+\epsilon}\log
x\right)  .
\]
Recall that the second-order parameter $\tau<0,$ and let $\epsilon$ be
sufficiently so that $\tau+\epsilon<0.$ Then it is easy to verify that the
absolute value of the latter integral is less than or equal to%
\[
\int_{1}^{\infty}\frac{\overline{F}\left(  ux\right)  }{\overline{F}\left(
u\right)  }\left(  1-\left(  \tau+\epsilon\right)  \log x\right)
x^{\tau+\epsilon-1}dx.
\]
We have $0<\overline{F}\left(  ux\right)  /\overline{F}\left(  u\right)
\leq1,$ over $x\geq1,$ and%
\[
0<\int_{1}^{\infty}x^{\tau+\epsilon-1}\left(  1-\left(  \tau+\epsilon\right)
\log x\right)  dx=-2/\left(  \tau+\epsilon\right)  <\infty,
\]
therefore $I_{u,1}=o\left(  1\right)  ,$ as $u\rightarrow\infty.$ Likewise, by
using an integration by part and then by applying once again the previous
Potter's inequality, we show that $I_{u,2}=o\left(  1\right)  $ as well. By
using the change of variable $x^{-1/\gamma},$ yields $I_{u,3}=-\gamma\int%
_{0}^{1}J\left(  s\right)  \left(  \log s\right)  ds,$ which completes the
proof of the first assertion. To prove the second one, let us write%
\[
\mathbf{I}_{u}:=%
{\displaystyle\int_{1}^{\infty}}
J\left(  \frac{\overline{F}\left(  ux\right)  }{\overline{F}\left(  u\right)
}\right)  \dfrac{d}{d\gamma}\log\ell_{\gamma,J}\left(  x\right)
d\frac{F\left(  ux\right)  }{\overline{F}\left(  u\right)  }.
\]
Recall that $\ell_{\gamma,J}\left(  x\right)  =J\left(  x^{-1/\gamma}\right)
\ell_{\gamma}\left(  x\right)  $ and decompose $\mathbf{I}_{u}$ into the sum
of
\[
\mathbf{I}_{u}^{\left(  1\right)  }:=%
{\displaystyle\int_{1}^{\infty}}
J\left(  \frac{\overline{F}\left(  ux\right)  }{\overline{F}\left(  u\right)
}\right)  \dfrac{d}{d\gamma}\log J\left(  x^{-1/\gamma}\right)  d\frac
{F\left(  ux\right)  }{\overline{F}\left(  u\right)  }%
\]
and%
\[
\mathbf{I}_{u}^{\left(  2\right)  }:=%
{\displaystyle\int_{1}^{\infty}}
J\left(  \frac{\overline{F}\left(  ux\right)  }{\overline{F}\left(  u\right)
}\right)  \dfrac{d}{d\gamma}\log\ell_{\gamma}\left(  x\right)  d\frac{F\left(
ux\right)  }{\overline{F}\left(  u\right)  }.
\]
It is clear that%
\[
\mathbf{I}_{u}^{\left(  1\right)  }=%
{\displaystyle\int_{1}^{\infty}}
\dfrac{d}{d\gamma}\log J\left(  x^{-1/\gamma}\right)  d\varphi\left(
\frac{\overline{F}\left(  ux\right)  }{\overline{F}\left(  u\right)  }\right)
,
\]
where $\varphi\left(  s\right)  :=\int_{s}^{1}J\left(  s\right)  dt.$ Let us
write
\[
\dfrac{d}{d\gamma}\log J\left(  x^{-1/\gamma}\right)  =-\gamma^{-1}%
x^{-1/\gamma}\mathcal{L}\left(  x^{-1/\gamma}\right)  ,
\]
where $\mathcal{L}\left(  s\right)  :=\left(  J^{\prime}\left(  s\right)
/J\left(  s\right)  \right)  \log s,$ thus%
\[
\mathbf{I}_{u}^{\left(  1\right)  }=-\gamma^{-1}%
{\displaystyle\int_{1}^{\infty}}
x^{-1/\gamma}\mathcal{L}\left(  x^{-1/\gamma}\right)  d\varphi\left(
\frac{\overline{F}\left(  ux\right)  }{\overline{F}\left(  u\right)  }\right)
.
\]
From assumption $\left[  A3\right]  $ the function $\mathcal{L}$ is bounded
over $\left(  0,1\right)  ,$ then by using an integration by parts, yields%
\[
\gamma\mathbf{I}_{u}^{\left(  1\right)  }=%
{\displaystyle\int_{1}^{\infty}}
\varphi\left(  \frac{\overline{F}\left(  ux\right)  }{\overline{F}\left(
u\right)  }\right)  d\left\{  x^{-1/\gamma}\mathcal{L}\left(  x^{-1/\gamma
}\right)  \right\}  .
\]
By using the routine manipulations of Potter's inequality as used above, we
infer that%
\[
\gamma\mathbf{I}_{u}^{\left(  1\right)  }=%
{\displaystyle\int_{1}^{\infty}}
\varphi\left(  x^{-1/\gamma}\right)  d\left\{  x^{-1/\gamma}\mathcal{L}\left(
x^{-1/\gamma}\right)  \right\}  +o\left(  1\right)  ,
\]
which by using the change of variables $s=x^{-1/\gamma},$ becomes $-%
{\displaystyle\int_{0}^{1}}
\varphi\left(  s\right)  d\left\{  s\mathcal{L}\left(  s\right)  \right\}  .$
Once again, by using twice integration by parts, with the fact that $%
{\displaystyle\int_{0}^{1}}
J\left(  s\right)  ds=1,$ we write
\begin{align*}
-%
{\displaystyle\int_{0}^{1}}
\varphi\left(  s\right)  d\left\{  s\mathcal{L}\left(  s\right)  \right\}   &
=-%
{\displaystyle\int_{0}^{1}}
J\left(  s\right)  s\mathcal{L}\left(  s\right)  ds=-%
{\displaystyle\int_{0}^{1}}
sJ^{\prime}\left(  s\right)  \log sds\\
&  =-\beta+1,
\end{align*}
therefore $\mathbf{I}_{u}^{\left(  1\right)  }=\gamma^{-1}-\gamma^{-1}%
\beta+o\left(  1\right)  .$ Let us now consider the second term $\mathbf{I}%
_{u}^{\left(  2\right)  },$ which equals%
\[%
{\displaystyle\int_{1}^{\infty}}
J\left(  \frac{\overline{F}\left(  ux\right)  }{\overline{F}\left(  u\right)
}\right)  \dfrac{d}{d\gamma}\log\ell_{\gamma}\left(  x\right)  d\frac{F\left(
ux\right)  }{\overline{F}\left(  u\right)  }.
\]
It is easy to verify that $d\log\ell_{\gamma}\left(  x\right)  /d\gamma
=\gamma^{-2}\left(  \log x-\gamma\right)  ,$ then
\[
\mathbf{I}_{u}^{\left(  2\right)  }=\gamma^{-2}%
{\displaystyle\int_{1}^{\infty}}
J\left(  \frac{\overline{F}\left(  ux\right)  }{\overline{F}\left(  u\right)
}\right)  \left(  \log x-\gamma\right)  d\frac{F\left(  ux\right)  }%
{\overline{F}\left(  u\right)  },
\]
which equals%
\[
\gamma^{-2}%
{\displaystyle\int_{1}^{\infty}}
J\left(  \frac{\overline{F}\left(  ux\right)  }{\overline{F}\left(  u\right)
}\right)  \log xd\frac{F\left(  ux\right)  }{\overline{F}\left(  u\right)
}-\gamma^{-1}%
{\displaystyle\int_{1}^{\infty}}
J\left(  \frac{\overline{F}\left(  ux\right)  }{\overline{F}\left(  u\right)
}\right)  d\frac{F\left(  ux\right)  }{\overline{F}\left(  u\right)  }.
\]
Note that the second integral equals$%
{\displaystyle\int_{0}^{1}}
J\left(  s\right)  ds=1,$ therefore
\[
\mathbf{I}_{u}^{\left(  2\right)  }=\gamma^{-2}%
{\displaystyle\int_{1}^{\infty}}
J\left(  \frac{\overline{F}\left(  ux\right)  }{\overline{F}\left(  u\right)
}\right)  \log xd\frac{F\left(  ux\right)  }{\overline{F}\left(  u\right)
}-\gamma^{-1}+o\left(  1\right)  .
\]
In summary, we showed that
\[
\mathbf{I}_{u}=\mathbf{I}_{u}^{\left(  1\right)  }+\mathbf{I}_{u}^{\left(
2\right)  }=\gamma^{-2}%
{\displaystyle\int_{1}^{\infty}}
\log xdF_{u}\left(  x\right)  -\gamma^{-1}\beta+o\left(  1\right)  ,\text{ as
}u\rightarrow\infty,
\]
which meets the second assertion. It is clear that $\beta^{-1}%
{\displaystyle\int_{1}^{\infty}}
\log xdF_{u}\left(  x\right)  \rightarrow\gamma$ is equivalent to
$\mathbf{I}_{u}\rightarrow0,$ thus we have completed the proof of Lemma
$\ref{lemma1}.$
\end{proof}

\begin{lemma}
\textbf{\label{lemma2}}Assume that $\overline{F}$ satisfies the first-order
condition $\left(  \ref{RV}\right)  .$ Let $J$ be a continuous function
fulfilling assumptions $\left[  A1\right]  -\left[  A3\right]  $ and let $k$
be a sequence of integers such that $k\rightarrow\infty$ and $k/n\rightarrow
0,$ then
\[
\pi_{k}^{\left(  1\right)  }\left(  \gamma_{0}\right)  \overset{\mathbf{P}%
}{\rightarrow}0\text{ and \ }\pi_{k}^{\left(  2\right)  }\left(  \gamma
_{0}\right)  \overset{\mathbf{P}}{\rightarrow}\eta_{\gamma_{0},\alpha},\text{
as }n\rightarrow\infty,
\]
for every $\alpha>0,$ where $\eta_{\gamma_{0},\alpha}$ is as in Theorem
$\ref{Theorem2}.$
\end{lemma}

\begin{proof}
Recall the formulas of $A_{k}^{\left(  1\right)  }\left(  \gamma_{0}\right)
,$ $\Psi_{\gamma_{0},\alpha}^{\left(  1\right)  }\left(  x\right)  ,$ $\pi
_{k}^{\left(  1\right)  }\left(  \gamma_{0}\right)  $ and $\ell_{\gamma_{0},J}
$ given respectively in $\left(  \ref{Am}\right)  ,$ $\left(  \ref{psi-m}%
\right)  ,$ $\left(  \ref{pi}\right)  $ and $\left(  \ref{l-J-gamma}\right)
.$ Let us rewrite $A_{k}^{\left(  1\right)  }\left(  \gamma_{0}\right)  $ into%
\[
\int_{1}^{\infty}J\left(  \frac{\overline{F}_{n}\left(  X_{n-k:n}x\right)
}{\overline{F}_{n}\left(  X_{n-k:n}\right)  }\right)  \Psi_{\gamma_{0},\alpha
}^{\left(  1\right)  }\left(  x\right)  d\frac{F_{n}\left(  X_{n-k:n}x\right)
}{\overline{F}_{n}\left(  X_{n-k:n}\right)  },
\]
where $F_{n}$ being the empirical cdf pertaining to the sample $X_{1}%
,...,X_{n}.$ It is clear that%
\[
A_{k}^{\left(  1\right)  }\left(  \gamma_{0}\right)  =-\int_{1}^{\infty}%
\Psi_{\gamma_{0},\alpha}^{\left(  1\right)  }\left(  x\right)  d\varphi\left(
\frac{\overline{F}_{n}\left(  X_{n-k:n}x\right)  }{\overline{F}_{n}\left(
X_{n-k:n}\right)  }\right)  ,
\]
where $\varphi\left(  v\right)  :=\int_{0}^{v}J\left(  t\right)
dt.$\textbf{\ }From Proposition $\ref{prop3},$\textbf{\ }we infer
that\textbf{\ }$\Psi_{\gamma_{0,\alpha}}^{\left(  1\right)  }\left(
\infty\right)  $ is finite, then by using an integration by parts yields%
\[
A_{k}^{\left(  1\right)  }\left(  \gamma_{0}\right)  =\Psi_{\gamma_{0},\alpha
}^{\left(  1\right)  }\left(  1\right)  +\int_{1}^{\infty}\varphi\left(
\frac{\overline{F}_{n}\left(  X_{n-k:n}x\right)  }{\overline{F}_{n}\left(
X_{n-k:n}\right)  }\right)  d\Psi_{\gamma_{0},\alpha}^{\left(  1\right)
}\left(  x\right)  .
\]
In Proposition\textbf{\ }$\ref{prop1},$ we stated that
\begin{align*}
\int_{1}^{\infty}\Psi_{\gamma_{0},\alpha+1}^{\left(  1\right)  }dx  &
=\left(  1+\frac{1}{\alpha}\right)  \int_{1}^{\infty}\ell_{\gamma_{0}%
,J}\left(  x\right)  \Psi_{\gamma_{0},\alpha}^{\left(  1\right)  }\left(
x\right)  dx\\
&  =\left(  1+\frac{1}{\alpha}\right)  \int_{1}^{\infty}\Psi_{\gamma
_{0},\alpha}^{\left(  1\right)  }\left(  x\right)  \gamma_{0}^{-1}%
x^{-1-1/\gamma_{0}}J\left(  x^{-1/\gamma_{0}}\right)  dx,
\end{align*}
which may rewritten into%
\[
\left(  1+\frac{1}{\alpha}\right)  \int_{1}^{\infty}\Psi_{\gamma_{0},\alpha
}^{\left(  1\right)  }\left(  x\right)  J\left(  x^{-1/\gamma_{0}}\right)
dx^{-1/\gamma_{0}}=-\left(  1+\frac{1}{\alpha}\right)  \int_{1}^{\infty}%
\Psi_{\gamma_{0},\alpha}^{\left(  1\right)  }\left(  x\right)  d\varphi\left(
x^{-1/\gamma_{0}}\right)  ,
\]
thus by using an integration by parts yields
\[
\int_{1}^{\infty}\Psi_{\gamma_{0},\alpha+1}^{\left(  1\right)  }dx=\left(
1+\frac{1}{\alpha}\right)  \left\{  \Psi_{\gamma_{0},\alpha}^{\left(
1\right)  }\left(  1\right)  +\int_{1}^{\infty}\varphi\left(  x^{-1/\gamma
_{0}}\right)  d\Psi_{\gamma_{0},\alpha}^{\left(  1\right)  }\left(  x\right)
\right\}  .
\]
It is easy to verify now that%
\[
-\left(  1+\frac{1}{\alpha}\right)  ^{-1}\pi_{k}^{\left(  1\right)  }\left(
\gamma_{0}\right)  =\int_{1}^{\infty}\left\{  \varphi\left(  \frac{n}%
{k}\overline{F}_{n}\left(  X_{n-k:n}x\right)  \right)  -\varphi\left(
\frac{n}{k}\overline{F}\left(  X_{n-k:n}x\right)  \right)  \right\}
d\Psi_{\gamma_{0},\alpha}^{\left(  1\right)  }\left(  x\right)  .
\]
By applyig Taylor's expansion to function $\varphi,$ we decompose the previous
expression into the sum of%
\[
T_{k}^{\left(  1\right)  }:=\int_{1}^{\infty}\left(  \frac{n}{k}\overline
{F}_{n}\left(  X_{n-k:n}x\right)  -\frac{n}{k}\overline{F}\left(
X_{n-k:n}x\right)  \right)  J\left(  x^{-1/\gamma_{0}}\right)  d\Psi
_{\gamma_{0},\alpha}^{\left(  1\right)  }\left(  x\right)
\]
and%
\[
R_{k}^{\left(  1\right)  }:=\frac{1}{2}\int_{1}^{\infty}\left(  \frac{n}%
{k}\overline{F}_{n}\left(  xX_{n-k:n}\right)  -\frac{n}{k}\overline{F}\left(
xX_{n-k:n}\right)  \right)  ^{2}J^{\left(  1\right)  }\left(  d_{n}\left(
x\right)  \right)  d\Psi_{\gamma_{0},\alpha}^{\left(  1\right)  }\left(
x\right)  ,
\]
where $d_{n}\left(  x\right)  $ is between $\frac{n}{k}\overline{F}\left(
xX_{n-k:n}\right)  $ and $\frac{n}{k}\overline{F}_{n}\left(  xX_{n-k:n}%
\right)  $ and $J^{\left(  1\right)  }$ denotes the first derivative of $J.$
From the first assertion of Proposition $\ref{prop4},$ we deduce that: for
every $0<\epsilon<1/2$%
\[
\frac{n}{k}\overline{F}_{n}\left(  X_{n-k:n}x\right)  -\frac{n}{k}\overline
{F}\left(  X_{n-k:n}x\right)  =k^{-1/2}O_{\mathbf{P}}\left(  x^{\left(
\epsilon-1/2\right)  /\gamma_{0}}\right)  ,
\]
uniformly over $x>1.$ Since $k^{-1/2}\rightarrow0,$ then $k^{-1/2}%
O_{\mathbf{P}}\left(  x^{\left(  \epsilon-1/2\right)  /\gamma_{0}}\right)
=o_{\mathbf{P}}\left(  x^{\left(  \epsilon-1/2\right)  /\gamma_{0}}\right)  ,$
therefore%
\[
T_{k}^{\left(  1\right)  }=o_{\mathbf{P}}\left(  1\right)  \int_{1}^{\infty
}J\left(  x^{-1/\gamma_{0}}\right)  x^{\left(  \epsilon-1/2\right)
/\gamma_{0}}\left\vert \frac{d}{dx}\Psi_{\gamma_{0},\alpha}^{\left(  1\right)
}\left(  x\right)  \right\vert dx.
\]
It is clear that, for every $0<\epsilon<1/2,$ the function $x\rightarrow
x^{\left(  \epsilon-1/2\right)  /\gamma_{0}}J\left(  x^{-1/\gamma_{0}}\right)
$ is bounded over $\left(  0,1\right)  $ and from Proposition $\ref{prop3}$ we
have $\int_{1}^{\infty}\left\vert \frac{d}{dx}\Psi_{\gamma_{0},\alpha
}^{\left(  1\right)  }\left(  x\right)  \right\vert dx$ is finite, it follows
that $T_{k}^{\left(  1\right)  }=o_{\mathbf{P}}\left(  1\right)  .$ By using
similar arguments we show that $R_{k}^{\left(  1\right)  }=o_{\mathbf{P}%
}\left(  1\right)  $ as well. The proof of Lemma $\ref{lemma2}$ is now completed.
\end{proof}

\begin{lemma}
\textbf{\label{lemma3}}Assume that $\overline{F}$ satisfies the second-order
condition $\left(  \ref{second-order}\right)  .$ Let $J$ be a continuous
function fulfilling assumptions $\left[  A1\right]  -\left[  A3\right]  $ and
$k$ be a sequence of integers such that $k\rightarrow\infty,$ $k/n\rightarrow0
$ and $\sqrt{k}a\left(  n/k\right)  $ is asymptotically bounded. Then, in the
probability space $\left(  \Omega,\mathcal{A},\mathbf{P}\right)  $ there
exists a standard Wiener process $\left\{  W\left(  x\right)  ,\text{ }%
x\geq0\right\}  ,$ such that for every $\alpha>0:$%
\begin{align}
&  -\left(  1+\frac{1}{\alpha}\right)  ^{-1}\sqrt{k}\pi_{k}^{\left(  1\right)
}\left(  \gamma_{0}\right) \nonumber\\
&  =\int_{1}^{\infty}\left(  W\left(  x^{-1/\gamma_{0}}\right)  -x^{-1/\gamma
_{0}}W\left(  1\right)  \right)  J\left(  x^{-1/\gamma_{0}}\right)
d\Psi_{\gamma_{0},\alpha}^{\left(  1\right)  }\left(  x\right)  +\sqrt
{k}a\left(  n/k\right)  B_{\gamma_{0},\alpha}^{\left(  1\right)
}+o_{\mathbf{P}}\left(  1\right)  , \label{app1}%
\end{align}
and%
\begin{align}
&  -\left(  1+\frac{1}{\alpha}\right)  ^{-1}\sqrt{k}\left(  \pi_{k}^{\left(
2\right)  }\left(  \gamma_{0}\right)  -\eta_{\gamma_{0},\alpha}\right)
\label{app2}\\
&  =\int_{1}^{\infty}\left(  W\left(  x^{-1/\gamma_{0}}\right)  -x^{-1/\gamma
_{0}}W\left(  1\right)  \right)  J\left(  x^{-1/\gamma_{0}}\right)
d\Psi_{\gamma_{0},\alpha}^{\left(  2\right)  }\left(  x\right)  +\sqrt
{k}a\left(  n/k\right)  B_{\gamma_{0},\alpha}^{\left(  2\right)
}+o_{\mathbf{P}}\left(  1\right)  ,\nonumber
\end{align}
as $n\rightarrow\infty,$ where
\begin{equation}
\eta_{\gamma_{0},\alpha}:=\left(  1+\alpha\right)  \int_{1}^{\infty}\left(
\Psi_{\gamma_{0},1}^{\left(  1\right)  }\left(  x\right)  \right)  ^{2}%
\ell_{\gamma_{0},J}^{\alpha-1}\left(  x\right)  dx, \label{eta}%
\end{equation}%
\begin{equation}
B_{\gamma_{0},\alpha}^{\left(  2\right)  }:=\int_{1}^{\infty}x^{-1/\gamma_{0}%
}\frac{x^{\tau/\gamma_{0}}-1}{\tau\gamma_{0}}J\left(  x^{-1/\gamma_{0}%
}\right)  d\Psi_{\gamma_{0},\alpha}^{\left(  2\right)  }\left(  x\right)  ,
\label{AB}%
\end{equation}
and $\Psi_{\gamma_{0},\alpha}^{\left(  m\right)  }$ is as in $\left(
\ref{psi-m}\right)  .$ Moreover
\begin{equation}
\pi_{k}^{\left(  1\right)  }\left(  \gamma_{0}\right)  \overset{\mathbf{P}%
}{\rightarrow}0\text{ and }\pi_{k}^{\left(  2\right)  }\left(  \gamma
_{0}\right)  \overset{\mathbf{P}}{\rightarrow}\eta_{\gamma_{0},\alpha},\text{
as }n\rightarrow\infty. \label{consist}%
\end{equation}

\end{lemma}

\begin{proof}
In the proof of Lemma $\ref{lemma2},$ we decomposed $-\left(  1+\frac
{1}{\alpha}\right)  ^{-1}\pi_{k}^{\left(  1\right)  }\left(  \gamma
_{0}\right)  $ into the sum of%
\[
T_{k}^{\left(  1\right)  }:=\int_{1}^{\infty}\left(  \frac{n}{k}\overline
{F}_{n}\left(  X_{n-k:n}x\right)  -\frac{n}{k}\overline{F}\left(
X_{n-k:n}x\right)  \right)  J\left(  x^{-1/\gamma_{0}}\right)  d\Psi
_{\gamma_{0},\alpha}^{\left(  1\right)  }\left(  x\right)
\]
and%
\[
R_{k}^{\left(  1\right)  }:=\frac{1}{2}\int_{1}^{\infty}\left(  \frac{n}%
{k}\overline{F}_{n}\left(  X_{n-k:n}x\right)  -\frac{n}{k}\overline{F}\left(
X_{n-k:n}x\right)  \right)  ^{2}J^{\prime}\left(  d_{n}\left(  x\right)
\right)  d\Psi_{\gamma_{0},\alpha}^{\left(  1\right)  }\left(  x\right)  ,
\]
where $d_{n}\left(  x\right)  $ is between $x^{-1/\gamma_{0}}$ and $\frac
{n}{k}\overline{F}_{n}\left(  X_{n-k:n}x\right)  .$ Let us rewrite
$T_{k}^{\left(  1\right)  }$ into the sum of%
\[
T_{k}^{\left(  1,1\right)  }:=\int_{1}^{\infty}\left(  \frac{n}{k}\overline
{F}_{n}\left(  X_{n-k:n}x\right)  -\frac{n}{k}\overline{F}\left(
X_{n-k:n}x\right)  \right)  J\left(  x^{-1/\gamma_{0}}\right)  d\Psi
_{\gamma_{0},\alpha}^{\left(  1\right)  }\left(  x\right)  ,
\]%
\[
T_{k}^{\left(  1,2\right)  }:=\int_{1}^{\infty}\left(  \frac{n}{k}\overline
{F}\left(  X_{n-k:n}x\right)  -\overline{F}\left(  U\left(  n/k\right)
x\right)  \right)  J\left(  x^{-1/\gamma_{0}}\right)  d\Psi_{\gamma_{0}%
,\alpha}^{\left(  1\right)  }\left(  x\right)
\]
and%
\[
T_{k}^{\left(  1,3\right)  }:=\int_{1}^{\infty}\left(  \frac{n}{k}\overline
{F}\left(  U\left(  n/k\right)  x\right)  -x^{-1/\gamma_{0}}\right)  J\left(
x^{-1/\gamma_{0}}\right)  d\Psi_{\gamma_{0},\alpha}^{\left(  1\right)
}\left(  x\right)  .
\]
In Proposition $\ref{prop4}$ we announced that: on the probability space
$\left(  \Omega,\mathcal{A},\mathbf{P}\right)  ,$ there exists a standard
Wiener process $\left\{  W\left(  x\right)  ,\text{ }x\geq0\right\}  ,$ such
that for $x\geq1$ and $0<\epsilon<1/2$%
\begin{equation}
\sup_{x\geq1}x^{\left(  1/2-\epsilon\right)  }\left\vert D_{k}\left(
x\right)  -W\left(  x^{-1/\gamma_{0}}\right)  \right\vert \overset{\mathbf{P}%
}{\rightarrow}0, \label{wa}%
\end{equation}
as $n\rightarrow\infty,$ where $D_{k}\left(  x\right)  :=\sqrt{k}\left(
\frac{n}{k}\overline{F}_{n}\left(  xX_{n-k:n}\right)  -\frac{n}{k}\overline
{F}\left(  X_{n-k:n}x\right)  \right)  .$ To apply this Gaussian
approximation, we decompose $\sqrt{k}T_{k}^{\left(  1,1\right)  }$ into the
sum of
\begin{equation}
\sqrt{k}T_{k,1}^{\left(  1,1\right)  }:=\int_{1}^{\infty}W\left(
x^{-1/\gamma_{0}}\right)  J\left(  x^{-1/\gamma_{0}}\right)  d\Psi_{\gamma
_{0},\alpha}^{\left(  1\right)  }\left(  x\right)  , \label{w-approxi-1}%
\end{equation}
and
\[
\sqrt{k}T_{k,2}^{\left(  1,1\right)  }:=\int_{1}^{\infty}\left(  D_{k}\left(
x\right)  -W\left(  x^{-1/\gamma_{0}}\right)  \right)  J\left(  x^{-1/\gamma
_{0}}\right)  d\Psi_{\gamma_{0},\alpha}^{\left(  1\right)  }\left(  x\right)
.
\]
We keep the term $\sqrt{k}T_{k,1}^{\left(  1,1\right)  }$ and show that
$\sqrt{k}T_{k,2}^{\left(  1,1\right)  }\overset{\mathbf{P}}{\rightarrow}0$ as
$n\rightarrow\infty.$ Indeed, by applying the Gaussian approximation $\left(
\ref{wa}\right)  ,$ yields%
\[
\sqrt{k}T_{k,2}^{\left(  1,1\right)  }=o_{\mathbf{P}}\left(  1\right)
\int_{1}^{\infty}x^{-\left(  1/2-\epsilon\right)  /\gamma_{0}}J\left(
x^{-1/\gamma_{0}}\right)  \left\vert \frac{d}{dx}\Psi_{\gamma_{0},\alpha
}^{\left(  1\right)  }\left(  x\right)  \right\vert dx.
\]
In view of assumption $\left[  A2\right]  ,$ the function $x\rightarrow
x^{-\left(  1/2-\epsilon\right)  /\gamma_{0}}J\left(  x^{-1/\gamma_{0}%
}\right)  $ is bounded over $x>1.$ On the other hand, from Proposition
$\ref{prop3}$ the integral $\int_{1}^{\infty}\left\vert \frac{d}{dx}%
\Psi_{\gamma_{0},\alpha}^{\left(  1\right)  }\left(  x\right)  \right\vert dx$
is finite, it follows that $\sqrt{k}T_{k,2}^{\left(  1,1\right)
}=o_{\mathbf{P}}\left(  1\right)  .$ Let us now consider the second term
$T_{k}^{\left(  1,2\right)  }$ which in turn may be rewritten into the sum
\begin{align*}
&  T_{k,1}^{\left(  1,2\right)  }%
\begin{array}
[c]{c}%
:=
\end{array}
\int_{1}^{\infty}\frac{\overline{F}\left(  U\left(  n/k\right)  x\right)
}{\overline{F}\left(  U\left(  n/k\right)  \right)  }\left(  \frac
{\overline{F}\left(  X_{n-k:n}x\right)  }{\overline{F}\left(  U\left(
n/k\right)  x\right)  }-\left(  \frac{X_{n-k:n}}{U\left(  n/k\right)
}\right)  ^{-1/\gamma_{0}}\right) \\
&
\ \ \ \ \ \ \ \ \ \ \ \ \ \ \ \ \ \ \ \ \ \ \ \ \ \ \ \ \ \ \ \ \ \ \ \ \ \ \ \ \ \ \ \ \ \ \ \ \ \ \ \ \ \ \ \ \ \times
J\left(  x^{-1/\gamma_{0}}\right)  d\Psi_{\gamma_{0},\alpha}^{\left(
1\right)  }\left(  x\right)
\end{align*}
and%
\[
T_{k,2}^{\left(  1,2\right)  }:=\left(  \left(  \frac{X_{n-k:n}}{U\left(
n/k\right)  }\right)  ^{-1/\gamma_{0}}-1\right)  \int_{1}^{\infty}%
\frac{\overline{F}\left(  U\left(  n/k\right)  x\right)  }{\overline{F}\left(
U\left(  n/k\right)  \right)  }J\left(  x^{-1/\gamma_{0}}\right)
d\Psi_{\gamma_{0},\alpha}^{\left(  1\right)  }\left(  x\right)  .
\]
Next we show that $\sqrt{k}T_{k,1}^{\left(  1,2\right)  }=o_{\mathbf{P}%
}\left(  1\right)  $ and$\sqrt{k}T_{k,2}^{\left(  1,2\right)  }$ leads to a
second Gaussian rv. To this end, we will apply the second-order Potter's
inequality to $\overline{F}$ corresponding to $\left(  \ref{second-order}%
\right)  ,$ saying that: for any $0<\epsilon<1,$ there exists $n_{0}%
=n_{0}\left(  \epsilon\right)  ,$ such that for all $n>n_{0}$ and $x\geq1$%
\[
\left\vert \frac{\overline{F}\left(  tx\right)  /\overline{F}\left(  t\right)
-x^{-1/\gamma_{0}}}{A\left(  t\right)  }-x^{-1/\gamma_{0}}\dfrac
{x^{\tau/\gamma_{0}}-1}{\tau\gamma_{0}}\right\vert \leq\epsilon x^{-1/\gamma
_{0}+\tau/\gamma_{0}+\epsilon},
\]
see, e.g., Proposition 4 together with Remark 1 in \cite{HJ11}. Let us set
$x_{n}:=X_{n-k:n}/U\left(  n/k\right)  $ and $t_{x}:=U\left(  n/k\right)  x$
then write%
\begin{align*}
\sqrt{k}T_{k,1}^{\left(  1,2\right)  }  &  =\int_{1}^{\infty}\sqrt{k}A\left(
t_{x}\right)  \frac{\overline{F}\left(  U\left(  n/k\right)  x\right)
}{\overline{F}\left(  U\left(  n/k\right)  \right)  }\\
&  \times\left(  \frac{\frac{\overline{F}\left(  t_{x}x_{n}\right)
}{\overline{F}\left(  t_{x}\right)  }-x_{n}^{-1/\gamma_{0}}}{A\left(
t_{x}\right)  }-x_{n}^{-1/\gamma_{0}}\frac{x_{n}^{\tau/\gamma_{0}}-1}%
{\tau\gamma_{0}}\right)  J\left(  x^{-1/\gamma_{0}}\right)  \frac{d}{dx}%
\Psi_{\gamma_{0},\alpha}^{\left(  1\right)  }\left(  x\right)  dx\\
&  +\int_{1}^{\infty}\sqrt{k}A\left(  t_{x}\right)  \frac{\overline{F}\left(
U\left(  n/k\right)  x\right)  }{\overline{F}\left(  U\left(  n/k\right)
\right)  }x_{n}^{-1/\gamma_{0}}\frac{x_{n}^{\tau/\gamma_{0}}-1}{\tau\gamma
_{0}}J\left(  x^{-1/\gamma_{0}}\right)  \frac{d}{dx}\Psi_{\gamma_{0},\alpha
}^{\left(  1\right)  }\left(  x\right)  dx\\
&
\begin{array}
[c]{c}%
=:
\end{array}
L_{n1}+L_{n2}.
\end{align*}
By applying the aforementioned inequality, we easily show that
\[
L_{n1}=o_{\mathbf{P}}\left(  x_{n}^{-1/\gamma_{0}+\tau/\gamma_{0}+\epsilon
}\right)  \int_{1}^{\infty}\sqrt{k}A\left(  t_{x}\right)  \frac{\overline
{F}\left(  U\left(  n/k\right)  x\right)  }{\overline{F}\left(  U\left(
n/k\right)  \right)  }J\left(  x^{-1/\gamma_{0}}\right)  \left\vert \frac
{d}{dx}\Psi_{\gamma_{0},\alpha}^{\left(  1\right)  }\left(  x\right)
\right\vert dx.
\]
Next we use the first-order Potter's inequality to a regularly varying
function\textbf{\ }$V$ of negative index $\rho,$ that is: for every\textbf{\ }%
$0<\epsilon<1,$\textbf{\ }there exists $t_{0}=t_{0}\left(  \epsilon\right)  ,$
such that for $t>t_{0}$ and $x\geq1$%
\begin{equation}
\left\vert V\left(  tx\right)  /V\left(  t\right)  -x^{-\rho}\right\vert
<\epsilon x^{\rho+\epsilon}, \label{first-order-PotterF}%
\end{equation}
see for instance Proposition B.1.10 in \cite{deHF06} page 369. Note that
$X_{n-k:n}/U\left(  n/k\right)  \overset{\mathbf{P}}{\rightarrow}1,$ as
$n\rightarrow\infty,$ then by applying this inequality both to $\overline{F} $
and $\left\vert A\right\vert ,$ we get%
\[
L_{n1}=o_{\mathbf{P}}\left(  \sqrt{k}a\left(  n/k\right)  \right)  \int%
_{1}^{\infty}x^{-1/\gamma_{0}}J\left(  x^{-1/\gamma_{0}}\right)  \left\vert
\frac{d}{dx}\Psi_{\gamma_{0},\alpha}^{\left(  1\right)  }\left(  x\right)
\right\vert dx.
\]
where $a\left(  n/k\right)  :=A\left(  U\left(  n/k\right)  \right)  .$ We
already mentioned before that the previous integral is finite and on the other
hand $\sqrt{k}a\left(  n/k\right)  =O\left(  1\right)  ,$ it follows that
$L_{n1}=o_{\mathbf{P}}\left(  1\right)  .$ In view of the above two
inequalities, we show readily that
\begin{align*}
L_{n2}  &  =\left(  1+o_{\mathbf{P}}\left(  1\right)  \right)  \sqrt
{k}a\left(  n/k\right)  x_{n}^{-1/\gamma_{0}}\frac{x_{n}^{\tau/\gamma_{0}}%
-1}{\tau\gamma_{0}}\\
&  \ \ \ \ \ \ \ \ \ \ \ \ \ \ \ \ \ \ \ \ \ \ \times\int_{1}^{\infty
}x^{-1/\gamma_{0}}J\left(  x^{-1/\gamma_{0}}\right)  \left\vert \frac{d}%
{dx}\Psi_{\gamma_{0},\alpha}^{\left(  1\right)  }\left(  x\right)  \right\vert
dx.
\end{align*}
Since $x_{n}\overset{\mathbf{P}}{\rightarrow}1,$ as $n\rightarrow\infty,$ it
follows $x_{n}^{\tau/\gamma_{0}}-1\overset{\mathbf{P}}{\rightarrow}0,$
therefore $L_{n2}=o_{\mathbf{P}}\left(  1\right)  $ as well, thus $\sqrt
{k}T_{k,1}^{\left(  1,2\right)  }=o_{\mathbf{P}}\left(  1\right)  .$ By using
the routine manipulations of the first-order Potter's inequality above,
corresponding to $\overline{F},$ we end up with
\[
T_{k,2}^{\left(  1,2\right)  }=\left(  1+o_{\mathbf{P}}\left(  1\right)
\right)  \left(  \left(  \frac{X_{n-k:n}}{U\left(  n/k\right)  }\right)
^{-1/\gamma_{0}}-1\right)  \int_{1}^{\infty}x^{-1/\gamma_{0}}J\left(
x^{-1/\gamma_{0}}\right)  d\Psi_{\gamma_{0},\alpha}^{\left(  1\right)
}\left(  x\right)  .
\]
By adopting similar procedures as used in the proof of the first assertion of
Theorem 1 of\textbf{\ }\cite{BMN-2015} and given in \cite{BchMN16}%
\textbf{\ }(page 243)\textbf{, }we show that
\[
\sqrt{k}\left(  \frac{X_{n-k:n}}{U\left(  n/k\right)  }\right)  ^{-1/\gamma
_{0}}-1=-W\left(  k/n\right)  +o_{\mathbf{P}}\left(  1\right)  .
\]
In summary, until now, we managed to show that%
\begin{align*}
&  -\left(  1+\frac{1}{\alpha}\right)  ^{-1}\sqrt{k}\pi_{k}^{\left(  1\right)
}\left(  \gamma_{0}\right) \\
&  =\int_{1}^{\infty}\left(  W\left(  x^{-1/\gamma_{0}}\right)  -x^{-1/\gamma
_{0}}W\left(  k/n\right)  \right)  J\left(  x^{-1/\gamma_{0}}\right)
d\Psi_{\gamma_{0},\alpha}^{\left(  1\right)  }\left(  x\right)  +\sqrt{k}%
T_{k}^{\left(  1,3\right)  }+\sqrt{k}R_{k}^{\left(  1\right)  }.
\end{align*}
Next we take care of $\sqrt{k}T_{k}^{\left(  1,3\right)  }$ which represents
the asymptotic bias of the latter Gaussian approximation, while $\sqrt{k}%
R_{k}^{\left(  1\right)  }$ being an asymptotically negligible term. Note that
$\overline{F}\left(  U\left(  n/k\right)  \right)  =k/n$ and write%
\[
T_{k}^{\left(  1,3\right)  }=\int_{1}^{\infty}\left(  \frac{\overline
{F}\left(  U\left(  n/k\right)  x\right)  }{\overline{F}\left(  U\left(
n/k\right)  \right)  }-x^{-1/\gamma_{0}}\right)  J\left(  x^{-1/\gamma_{0}%
}\right)  d\Psi_{\gamma_{0},\alpha}^{\left(  1\right)  }\left(  x\right)  .
\]
Making use of inequality $\left(  \ref{AB}\right)  ,$ we end up with
\[
\sqrt{k}T_{k}^{\left(  1,3\right)  }=\left(  1+o_{\mathbf{P}}\left(  1\right)
\right)  \sqrt{k}a\left(  n/k\right)  B_{\gamma_{0},\alpha}^{\left(  1\right)
},
\]
where $B_{\gamma_{0},\alpha}^{\left(  1\right)  }$ is as in Theorem
$\ref{Theorem2}.$ Recall that the integral $B_{\gamma_{0},\alpha}^{\left(
1\right)  }$ is finite and $\sqrt{k}a\left(  n/k\right)  $ is asymptotically
bounded, thereby $\sqrt{k}T_{k}^{\left(  1,3\right)  }=\sqrt{k}a\left(
n/k\right)  B_{\gamma_{0},\alpha}^{\left(  1\right)  }+o_{\mathbf{P}}\left(
1\right)  .$ To finish observe that by using fact that $D_{k}\left(  x\right)
=O_{\mathbf{P}}\left(  1\right)  $ and the regularity assumptions on function
$J,$ we show that $\sqrt{k}R_{k}^{\left(  1\right)  }=o_{\mathbf{P}}\left(
1\right)  ,$ therefore we omit the details. In summary, we showed that%
\begin{align*}
&  -\left(  1+\frac{1}{\alpha}\right)  ^{-1}\sqrt{k}\pi_{k}^{\left(  1\right)
}\left(  \gamma_{0}\right) \\
&  =\int_{1}^{\infty}\left(  W\left(  x^{-1/\gamma_{0}}\right)  -x^{-1/\gamma
_{0}}W\left(  1\right)  \right)  J\left(  x^{-1/\gamma_{0}}\right)
d\Psi_{\gamma_{0},\alpha}^{\left(  1\right)  }\left(  x\right)  +\sqrt
{k}a\left(  n/k\right)  B_{\gamma_{0},\alpha}^{\left(  1\right)
}+o_{\mathbf{P}}\left(  1\right)  ,
\end{align*}
thus $\left(  \ref{app1}\right)  $ comes. Let us now prove assertion $\left(
\ref{app2}\right)  $. From Proposition $\ref{prop1},$ we have%
\[
\int_{1}^{\infty}\Psi_{\gamma_{0},\alpha+1}^{\left(  2\right)  }\left(
x\right)  dx=\left(  1+\frac{1}{\alpha}\right)  \int_{1}^{\infty}\ell
_{\gamma_{0}}\left(  x\right)  \Psi_{\gamma_{0},\alpha}^{\left(  2\right)
}\left(  x\right)  dx+\eta_{\gamma_{0},\alpha},
\]
it follows that%
\[
\left(  1+\frac{1}{\alpha}\right)  ^{-1}\left(  \pi_{k}^{\left(  2\right)
}\left(  \gamma_{0}\right)  -\eta_{\gamma_{0},\alpha}\right)  =\int%
_{1}^{\infty}\ell_{\gamma_{0}}\left(  x\right)  \Psi_{\gamma_{0},\alpha
}^{\left(  2\right)  }\left(  x\right)  dx-A_{k}^{\left(  2\right)  }\left(
\gamma_{0}\right)  .
\]
By using similar arguments as used in the proof of assertion $\left(
\ref{app1}\right)  ,$ we also show that the right side of the previous
equation equals
\[
\int_{1}^{\infty}\left(  W\left(  x^{-1/\gamma_{0}}\right)  -x^{-1/\gamma_{0}%
}W\left(  1\right)  \right)  J\left(  x^{-1/\gamma_{0}}\right)  d\Psi
_{\gamma_{0},\alpha}^{\left(  2\right)  }\left(  x\right)  +\lambda
B_{\gamma_{0},\alpha}^{\left(  2\right)  }+o_{\mathbf{P}}\left(  1\right)  ,
\]
thus $\left(  \ref{app2}\right)  $ holds too. To show assertion $\left(
\ref{consist}\right)  ,$ let us first note that
\begin{align*}
I_{\gamma_{0},\alpha}  &  :=\mathbf{E}\left\vert \int_{1}^{\infty}\left(
W\left(  x^{-1/\gamma_{0}}\right)  -x^{-1/\gamma_{0}}W\left(  1\right)
\right)  J\left(  x^{-1/\gamma_{0}}\right)  d\Psi_{\gamma_{0},\alpha}^{\left(
2\right)  }\left(  x\right)  \right\vert \\
&  \leq\int_{1}^{\infty}\mathbf{E}\left\vert W\left(  x^{-1/\gamma_{0}%
}\right)  -x^{-1/\gamma_{0}}W\left(  1\right)  \right\vert J\left(
x^{-1/\gamma_{0}}\right)  \left\vert \frac{d\Psi_{\gamma_{0},\alpha}^{\left(
2\right)  }\left(  x\right)  }{dx}\right\vert dx,
\end{align*}
Note that $B\left(  s\right)  :=W\left(  s\right)  -sW\left(  1\right)  $
being a Brownian bridge, then
\[
\mathbf{E}\left\vert B\left(  x^{-1/\gamma_{0}}\right)  \right\vert \leq
x^{-1/\left(  2\gamma_{0}\right)  }\leq1,\text{ for }x\geq1,
\]
From Proposition $\ref{prop3},$ $\int_{1}^{\infty}\left\vert d\Psi_{\gamma
_{0},\alpha}^{\left(  2\right)  }\left(  x\right)  /dx\right\vert dx$ is
finite, therefore%
\[
\sqrt{k}\pi_{k}^{\left(  1\right)  }\left(  \gamma_{0}\right)  =O_{\mathbf{P}%
}\left(  1\right)  =\sqrt{k}\left(  \pi_{k}^{\left(  2\right)  }\left(
\gamma_{0}\right)  -\eta_{\gamma_{0},\alpha}\right)  ,
\]
thus $\pi_{k}^{\left(  1\right)  }\left(  \gamma_{0}\right)  =o_{\mathbf{P}%
}\left(  1\right)  =\left(  \pi_{k}^{\left(  2\right)  }\left(  \gamma
_{0}\right)  -\eta_{\gamma_{0},\alpha}\right)  ,$ because $k^{-1}%
\rightarrow0,$ as $n\rightarrow\infty$ as sought.
\end{proof}

\begin{lemma}
\textbf{\label{lemma4}}Assume that assumptions $\left[  A1\right]  -\left[
A3\right]  ,$ then fo given a consistent estimator $\widehat{\gamma}$ of
$\gamma_{0},$ we have $\pi_{k}^{\left(  3\right)  }\left(  \widehat{\gamma
}\right)  =O_{\mathbf{P}}\left(  1\right)  ,$ provided that $\alpha>0.$
\end{lemma}

\begin{proof}
Let us first show that $\pi_{k}^{\left(  3\right)  }\left(  \gamma_{0}\right)
=O_{\mathbf{P}}\left(  1\right)  $ and recall that
\[
\pi_{k}^{\left(  3\right)  }\left(  \gamma_{0}\right)  =\int_{1}^{\infty}%
\Psi_{\gamma_{0},\alpha+1}^{\left(  3\right)  }\left(  x\right)  dx-\left(
1+\frac{1}{\alpha}\right)  \frac{1}{k}\sum_{i=1}^{k}J\left(  \frac{i}%
{k}\right)  \Psi_{\gamma_{0},\alpha}^{\left(  3\right)  }\left(
\frac{X_{n-i+1:n}}{X_{n-k:n}}\right)  .
\]
Making use of Proposition $\ref{prop1},$ we may rewrite $\pi_{k}^{\left(
3\right)  }\left(  \gamma_{0}\right)  $ into the sum of $\int_{1}^{\infty
}g_{\gamma_{0},\alpha}\left(  x\right)  dx$ and $\left(  1+\frac{1}{\alpha
}\right)  A_{k}^{\left(  3\right)  }\left(  \gamma_{0}\right)  ,$ where
\[
A_{k}^{\left(  3\right)  }\left(  \gamma_{0}\right)  :=\int_{1}^{\infty}%
\dfrac{d}{d\gamma}\ell_{\gamma_{0},J}\left(  x\right)  \Psi_{\gamma_{0}%
,\alpha}^{\left(  3\right)  }\left(  x\right)  dx-\frac{1}{k}\sum_{i=1}%
^{k}J\left(  \frac{i}{k}\right)  \Psi_{\gamma_{0},\alpha}^{\left(  3\right)
}\left(  \frac{X_{n-i+1:n}}{X_{n-k:n}}\right)  .
\]
By similar arguments as those used in the proof of Lemma $\ref{lemma2},$ we
also show that $A_{k}^{\left(  3\right)  }\left(  \gamma_{0}\right)
\overset{\mathbf{P}}{\rightarrow}0,$ so $\pi_{k}^{\left(  3\right)  }\left(
\gamma_{0}\right)  \overset{\mathbf{P}}{\rightarrow}\int_{1}^{\infty}%
g_{\gamma_{0},\alpha}\left(  x\right)  dx<\infty.$ From the third assertion of
Proposition $\ref{prop1},$ we have%
\[
g_{\gamma_{0},\alpha}\left(  x\right)  =\Psi_{\gamma_{0},\alpha+1}^{\left(
3\right)  }\left(  x\right)  -\left(  1+\frac{1}{\alpha}\right)  \ell
_{\gamma,J}\left(  x\right)  \Psi_{\gamma_{0},\alpha}^{\left(  3\right)
}\left(  x\right)  .
\]
By applying Proposition $\ref{prop2},$ we infer that $\int_{1}^{\infty
}\left\vert g_{\gamma_{0},\alpha}\left(  x\right)  \right\vert dx<\infty,$
thus $\pi_{k}^{\left(  3\right)  }\left(  \gamma_{0}\right)  =O_{\mathbf{P}%
}\left(  1\right)  .$ Next we prove that $\pi_{k}^{\left(  3\right)  }\left(
\widehat{\gamma}_{0}\right)  -\pi_{k}^{\left(  3\right)  }\left(  \gamma
_{0}\right)  =o_{\mathbf{P}}\left(  1\right)  ,$ as $n\rightarrow\infty.$
Indeed, let us write
\begin{align*}
&  \pi_{k}^{\left(  3\right)  }\left(  \widehat{\gamma}\right)  -\pi
_{k}^{\left(  3\right)  }\left(  \gamma_{0}\right)  \medskip\\
&  =\int_{1}^{\infty}\left\{  \Psi_{\widehat{\gamma},\alpha+1}^{\left(
3\right)  }\left(  x\right)  -\Psi_{\gamma_{0},\alpha+1}^{\left(  3\right)
}\left(  x\right)  \right\}  dx\medskip\\
&  -\left(  1+\frac{1}{\alpha}\right)  \frac{1}{k}\sum_{i=1}^{k}J\left(
\frac{i}{k}\right)  \left\{  \Psi_{\widehat{\gamma},\alpha}^{\left(  3\right)
}\left(  \frac{X_{n-i+1:n}}{X_{n-k:n}}\right)  -\Psi_{\gamma_{0},\alpha
}^{\left(  3\right)  }\left(  \frac{X_{n-i+1:n}}{X_{n-k:n}}\right)  \right\}
.
\end{align*}
Using the mean value theorem, yields%
\[
\int_{1}^{\infty}\left\{  \Psi_{\widehat{\gamma},\alpha+1}^{\left(  3\right)
}\left(  x\right)  -\Psi_{\gamma_{0},\alpha+1}^{\left(  3\right)  }\left(
x\right)  \right\}  dx=\left(  \widehat{\gamma}-\gamma_{0}\right)
\Psi_{\widehat{\gamma}_{0}^{\ast},\alpha+1}^{\left(  4\right)  }\left(
x\right)  \int_{1}^{\infty}\Psi_{\widehat{\gamma}_{0}^{\ast},\alpha
+1}^{\left(  4\right)  }\left(  x\right)  dx,
\]
where $\widehat{\gamma}_{0}^{\ast}$ is between $\widehat{\gamma}$ and
$\gamma.$ Recall that $\widehat{\gamma}\overset{\mathbf{P}}{\rightarrow}%
\gamma_{0}$ and by Proposition $\ref{prop2}$
\[
\int_{1}^{\infty}\left\vert \Psi_{\widehat{\gamma}_{0}^{\ast},\alpha
+1}^{\left(  4\right)  }\left(  x\right)  \right\vert dx=O_{\mathbf{P}}\left(
1\right)  ,
\]
thus $\int_{1}^{\infty}\left\{  \Psi_{\widehat{\gamma},\alpha+1}^{\left(
3\right)  }\left(  x\right)  -\Psi_{\gamma_{0},\alpha+1}^{\left(  3\right)
}\left(  x\right)  \right\}  dx=o_{\mathbf{P}}\left(  1\right)  .$ On the
other hand
\[%
\begin{array}
[c]{cl}
& \dfrac{1}{k}%
{\displaystyle\sum\limits_{i=1}^{k}}
J\left(  \dfrac{i}{k}\right)  \left\{  \Psi_{\widehat{\gamma},\alpha}^{\left(
3\right)  }\left(  \dfrac{X_{n-i+1:n}}{X_{n-k:n}}\right)  -\Psi_{\gamma
_{0},\alpha}^{\left(  3\right)  }\left(  \dfrac{X_{n-i+1:n}}{X_{n-k:n}%
}\right)  \right\}  \bigskip\\
= &
{\displaystyle\int_{1}^{\infty}}
J\left(  \dfrac{\overline{F}_{n}\left(  xX_{n-k:n}\right)  }{\overline{F}%
_{n}\left(  X_{n-k:n}\right)  }\right)  \left\{  \Psi_{\widehat{\gamma}%
,\alpha}^{\left(  3\right)  }\left(  x\right)  -\Psi_{\gamma_{0},\alpha
}^{\left(  3\right)  }\left(  x\right)  \right\}  d\dfrac{F_{n}\left(
xX_{n-k:n}\right)  }{\overline{F}_{n}\left(  X_{n-k:n}\right)  }.
\end{array}
\]
Once again making use of the mean value theorem, we write%
\[
\left(  \widehat{\gamma}-\gamma_{0}\right)  \int_{1}^{\infty}J\left(
\frac{\overline{F}_{n}\left(  xX_{n-k:n}\right)  }{\overline{F}_{n}\left(
X_{n-k:n}\right)  }\right)  \Psi_{\overline{\gamma}_{0}^{\ast},\alpha
}^{\left(  4\right)  }\left(  x\right)  d\frac{F_{n}\left(  xX_{n-k:n}\right)
}{\overline{F}_{n}\left(  X_{n-k:n}\right)  },
\]
where $\overline{\gamma}_{0}^{\ast}$ is between $\widehat{\gamma}$ and
$\gamma_{0}.$ From Proposition $\ref{prop2},$ the function $x\rightarrow
\Psi_{\overline{\gamma}_{0}^{\ast},\alpha}^{\left(  4\right)  }\left(
x\right)  $ being bounded over $x>1,$ then the previous quantity equals
$o_{\mathbf{P}}\left(  1\right)  \int_{0}^{1}J\left(  s\right)
ds=o_{\mathbf{P}}\left(  1\right)  .$ Thus we showed that $\pi_{k}^{\left(
3\right)  }\left(  \widehat{\gamma}\right)  -\pi_{k}^{\left(  3\right)
}\left(  \gamma_{0}\right)  =o_{\mathbf{P}}\left(  1\right)  ,$ leading to
$\pi_{k}^{\left(  3\right)  }\left(  \widehat{\gamma}\right)  =O_{\mathbf{P}%
}\left(  1\right)  $ as well.
\end{proof}

\begin{proposition}
\textbf{\label{prop1}}For every $\alpha>0,$ we have%
\[
\Psi_{\gamma,\alpha+1}^{\left(  1\right)  }\left(  x\right)  =\left(
1+\frac{1}{\alpha}\right)  \ell_{\gamma,J}\left(  x\right)  \Psi
_{\gamma,\alpha}^{\left(  1\right)  }\left(  x\right)  ,
\]%
\[
\Psi_{\gamma,\alpha+1}^{\left(  2\right)  }\left(  x\right)  =\left(
1+\frac{1}{\alpha}\right)  \ell_{\gamma,J}\left(  x\right)  \Psi
_{\gamma,\alpha}^{\left(  2\right)  }\left(  x\right)  +\left(  1+\alpha
\right)  \left(  \Psi_{\gamma,1}^{\left(  1\right)  }\left(  x\right)
\right)  ^{2}\ell_{\gamma,J}^{\alpha-1}\left(  x\right)
\]
and%
\[
\Psi_{\gamma,\alpha+1}^{\left(  3\right)  }\left(  x\right)  =\left(
1+\frac{1}{\alpha}\right)  \ell_{\gamma,J}\left(  x\right)  \Psi
_{\gamma,\alpha}^{\left(  3\right)  }\left(  x\right)  +g_{\gamma,\alpha
}\left(  x\right)  ,\text{ }%
\]
for some integrable function $x\rightarrow g_{\gamma,\alpha}\left(  x\right)
.$
\end{proposition}

\begin{proof}
Recall that $\Psi_{\gamma,\alpha}^{\left(  m\right)  }\left(  x\right)
:=d^{m}\ell_{\gamma,J}^{\alpha}\left(  x\right)  /d^{m}\gamma$ denotes the
values of $m$-th derivative of $\ell_{\gamma,J}^{\alpha},$ with respect to
$\gamma.$ The proof of the first identity is straightforward. Indeed%
\begin{align*}
\Psi_{\gamma,\alpha+1}^{\left(  1\right)  }\left(  x\right)   &  =\frac
{d}{d\gamma}\ell_{\gamma,J}^{\alpha+1}\left(  x\right)  =\left(
1+\alpha\right)  \ell_{\gamma,J}^{\alpha}\left(  x\right)  \frac{d}{d\gamma
}\ell_{\gamma,J}\left(  x\right)  \\
&  =\left(  1+\frac{1}{\alpha}\right)  \ell_{\gamma,J}\left(  x\right)
\frac{d}{d\gamma}\ell_{\gamma,J}^{\alpha}\left(  x\right)  =\left(  1+\frac
{1}{\alpha}\right)  \Psi_{\gamma,\alpha}^{\left(  1\right)  }\left(  x\right)
.
\end{align*}
For the second derivative, we compute:%
\begin{align*}
\Psi_{\gamma,\alpha+1}^{\left(  2\right)  }\left(  x\right)   &  =\frac{d^{2}%
}{d\gamma^{2}}\ell_{\gamma}^{\alpha+1}\left(  x\right)  =\frac{d}{d\gamma
}\left[  \left(  1+\alpha\right)  \frac{d}{d\gamma}\ell_{\gamma}\left(
x\right)  \ell_{\gamma}^{\alpha}\left(  x\right)  \right]  \\
&  =\left(  1+\alpha\right)  \left[  \frac{d^{2}}{d\gamma^{2}}\ell_{\gamma
}\left(  x\right)  \ell_{\gamma}^{\alpha}\left(  x\right)  +\alpha\left(
\frac{d}{d\gamma}\ell_{\gamma}\left(  x\right)  \right)  ^{2}\ell_{\gamma
}^{\alpha-1}\left(  x\right)  \right]  .
\end{align*}
On the other hand%
\[
\ell_{\gamma}\left(  x\right)  \frac{d^{2}}{d\gamma^{2}}\ell_{\gamma}^{\alpha
}\left(  x\right)  =\alpha\ell_{\gamma}\left(  x\right)  \left[  \frac{d^{2}%
}{d\gamma^{2}}\ell_{\gamma}\left(  x\right)  \ell_{\gamma}^{\alpha-1}\left(
x\right)  +\left(  \alpha-1\right)  \left(  \frac{d}{d\gamma}\ell_{\gamma
}\left(  x\right)  \right)  ^{2}\ell_{\gamma}^{\alpha-2}\left(  x\right)
\right]  ,
\]
it follows that%
\begin{align*}
&  \left(  1+\frac{1}{\alpha}\right)  \ell_{\gamma}\left(  x\right)
\frac{d^{2}}{d\gamma^{2}}\ell_{\gamma}^{\alpha}\left(  x\right)  \\
&  =\left(  \alpha+1\right)  \left[  \ell_{\gamma}\left(  x\right)
\frac{d^{2}}{d\gamma^{2}}\ell_{\gamma}\left(  x\right)  \ell_{\gamma}%
^{\alpha-1}\left(  x\right)  +\left(  \alpha-1\right)  \left(  \frac
{d}{d\gamma}\ell_{\gamma}\left(  x\right)  \right)  ^{2}\ell_{\gamma}%
^{\alpha-1}\left(  x\right)  \right]  .
\end{align*}
Thus%
\begin{align*}
\frac{d^{2}}{d\gamma^{2}}\ell_{\gamma}^{\alpha+1}\left(  x\right)   &
=\left(  1+\frac{1}{\alpha}\right)  \ell_{\gamma}\left(  x\right)  \frac
{d^{2}}{d\gamma^{2}}\ell_{\gamma}^{\alpha}\left(  x\right)  +\left(
1+\alpha\right)  \ell_{\gamma}^{\alpha-1}\left(  x\right)  \left(  \frac
{d}{d\gamma}\ell_{\gamma}\left(  x\right)  \right)  ^{2}\\
&  =\left(  1+\frac{1}{\alpha}\right)  \ell_{\gamma,J}\left(  x\right)
\Psi_{\gamma,\alpha}^{\left(  2\right)  }\left(  x\right)  +\left(
1+\alpha\right)  \left(  \Psi_{\gamma,1}^{\left(  1\right)  }\left(  x\right)
\right)  ^{2}\ell_{\gamma,J}^{\alpha-1}\left(  x\right)  ,
\end{align*}
which corresponds to the second assertion of the proposition. Let us now prove
the third assertion. It is easy to verify that:%
\begin{align*}
\Psi_{\gamma,\alpha+1}^{\left(  3\right)  }\left(  x\right)   &  =\frac{d^{3}%
}{d\gamma^{3}}\ell_{\gamma,J}^{\alpha+1}\left(  x\right)  =\left(  1+\frac
{1}{\alpha}\right)  \ell_{\gamma,J}\left(  x\right)  \frac{d^{3}}{d\gamma^{3}%
}\ell_{\gamma,J}^{\alpha}\left(  x\right)  +g_{\gamma}\left(  x\right)  \\
&  =\left(  1+\frac{1}{\alpha}\right)  \ell_{\gamma,J}\left(  x\right)
\Psi_{\gamma,\alpha}^{\left(  3\right)  }\left(  x\right)  +g_{\gamma}\left(
x\right)  ,
\end{align*}
where%
\[%
\begin{array}
[c]{l}%
g_{\gamma}\left(  x\right)  :=\left(  1+\dfrac{1}{\alpha}\right)  \dfrac
{d}{d\gamma}\ell_{\gamma,J}\left(  x\right)  \dfrac{d^{2}}{d\gamma^{2}}%
\ell_{\gamma,J}^{\alpha}\left(  x\right)  \medskip\\
\ \ \ \ \ \ \ \ \ +\left(  1+\alpha\right)  \ell_{\gamma,J}^{\alpha-2}\left(
x\right)  \dfrac{d}{d\gamma}\ell_{\gamma,J}\left(  x\right)  \left\{
2\ell_{\gamma,J}\left(  x\right)  \frac{d^{2}}{d\gamma^{2}}\ell_{\gamma
,J}^{\alpha}\left(  x\right)  +\left(  \alpha-1\right)  \left(  \dfrac
{d}{d\gamma}\ell_{\gamma,J}\left(  x\right)  \right)  ^{2}\right\}  .
\end{array}
\]
In Lemma $\ref{prop2},$ we established that both integrals
\[
\int_{1}^{\infty}\left\vert \Psi_{\gamma,\alpha+1}^{\left(  3\right)  }\left(
x\right)  \right\vert dx\text{ and }\int_{1}^{\infty}\left\vert \ell
_{\gamma,J}\left(  x\right)  \Psi_{\gamma,\alpha}^{\left(  3\right)  }\left(
x\right)  \right\vert dx
\]
are finite. On the other hand, from the expression%
\[
g_{\gamma}\left(  z\right)  =\Psi_{\gamma,\alpha+1}^{\left(  3\right)
}\left(  x\right)  -\left(  1+\frac{1}{\alpha}\right)  \ell_{\gamma,J}\left(
x\right)  \Psi_{\gamma,\alpha}^{\left(  3\right)  }\left(  x\right)  ,
\]
it follows that $\int_{1}^{\infty}\left\vert g_{\gamma}\left(  x\right)
\right\vert dx<\infty$ as well.
\end{proof}

\begin{proposition}
\textbf{\label{prop2}}Assume that assumption $\left[  A1\right]  -\left[
A3\right]  $ hold, then for every $\alpha>0,$ each derivative function
$\Psi_{\gamma,\alpha}^{\left(  m\right)  },$ $1\leq m\leq4,$ is a finite
linear combination of
\begin{equation}
E_{m,\alpha}^{\left(  i\right)  }\left(  x;\gamma\right)  :=\left(  x^{-a_{i}%
}\log^{b_{i}}x\right)  \ell_{\gamma,J}^{\alpha}\left(  x\right)
{\textstyle\prod_{j=1}^{m}}
\left[  \mathcal{L}^{\left(  j-1\right)  }\left(  x^{-1/\gamma}\right)
\right]  ^{c_{i,j}}, \label{lc}%
\end{equation}
for some constant $a_{i}\geq0$ and integers $b_{i},c_{i,j}\geq0,$ where
$\mathcal{L}^{\left(  j\right)  }$ denotes the $j$-th derivative of
$\mathcal{L} $ and $\mathcal{L}^{\left(  0\right)  }:=\mathcal{L}.$ Moreover
$\Psi_{\gamma,\alpha}^{\left(  m\right)  }$ is bounded over $x>1,$ and both
$\int_{1}^{\infty}\left\vert \Psi_{\gamma,\alpha+1}^{\left(  m\right)
}\left(  x\right)  \right\vert dx$ and $\int_{1}^{\infty}\left\vert
\ell_{\gamma,J}\left(  x\right)  \Psi_{\gamma,\alpha}^{\left(  m\right)
}\left(  x\right)  \right\vert dx$ are finite, for $1\leq m\leq4.$
\end{proposition}

\begin{proof}
Let us reasoning by recurrence. By using elementary derivative, yields%
\begin{align*}
&  \Psi_{\gamma,\alpha}^{\left(  1\right)  }\left(  x\right) \\
&  =\alpha\gamma^{-2}\ell_{\gamma,J}^{\alpha-1}\left(  x\right)  \left(
x^{-\frac{1}{\gamma}}J^{\prime}\left(  x^{-1/\gamma}\right)  \ell_{\gamma
}\left(  x\right)  \log x+x^{-\frac{1}{\gamma}-1}J\left(  x^{-1/\gamma
}\right)  \left(  \gamma^{-1}\log x-1\right)  \right)  ,
\end{align*}
which equals $\alpha\gamma^{-2}\ell_{\gamma,J}^{\alpha}\left(  x\right)
\left(  x^{-\frac{1}{\gamma}}\mathcal{L}\left(  x^{-1/\gamma}\right)
+\gamma^{-1}\log x-1\right)  $ where $\mathcal{L}$ is as in assumption
$\left[  A3\right]  .$ It is clear that the latter quantity may be rewritten
into
\begin{equation}
\alpha\gamma^{-2}\left\{  x^{-\frac{1}{\gamma}}\ell_{\gamma,J}^{\alpha}\left(
x\right)  \mathcal{L}\left(  x^{-1/\gamma}\right)  \right\}  +\alpha
\gamma^{-3}\left\{  \ell_{\gamma,J}^{\alpha}\left(  x\right)  \log x\right\}
-\alpha\gamma^{-2}\left\{  \ell_{\gamma,J}^{\alpha}\left(  x\right)  \right\}
, \label{la}%
\end{equation}
which a linear combination of expression $\left(  \ref{lc}\right)  ,$\ for
some suitable constants $a_{i}\geq0$ and integers $b_{i},c_{i,j}\geq0.$ Let us
assume that the result holds at order $m$ and show it is at order $m+1 $ as
well. To \ this end, we compute the derivative of expression $\left(
\ref{lc}\right)  .$ We have
\[
\frac{d}{d\gamma}E_{m,\alpha}^{\left(  i\right)  }\left(  x;\gamma\right)
=x^{-a_{i}}\log^{b_{i}}x\frac{d}{d\gamma}\left(  \ell_{\gamma,J}^{\alpha
}\left(  x\right)
{\textstyle\prod_{j=1}^{m}}
\left[  \mathcal{L}^{\left(  j-1\right)  }\left(  x^{-1/\gamma}\right)
\right]  ^{c_{i,j}}\right)  ,
\]
which equals%
\begin{align}
&  x^{-a_{i}}\log^{b_{i}}x%
{\textstyle\prod_{j=1}^{m}}
\left[  \mathcal{L}^{\left(  j-1\right)  }\left(  x^{-1/\gamma}\right)
\right]  ^{c_{i,j}}\frac{d}{d\gamma}\ell_{\gamma,J}^{\alpha}\left(  x\right)
+x^{-a_{i}}\log^{b_{i}}x\ell_{\gamma,J}^{\alpha}\left(  x\right)
\label{latter-exp}\\
&
\ \ \ \ \ \ \ \ \ \ \ \ \ \ \ \ \ \ \ \ \ \ \ \ \ \ \ \ \ \ \ \ \ \ \ \ \ \ \ \times
\gamma^{2}x^{-\frac{1}{\gamma}}\log x\left.  \frac{d}{dy}%
{\textstyle\prod_{j=1}^{m}}
\left[  \mathcal{L}^{\left(  j-1\right)  }\left(  y\right)  \right]
^{c_{i,j}}\right\vert _{y=x^{-1/\gamma}}.\nonumber
\end{align}
By using the product rule for derivatives, we write
\[
\frac{d}{dy}%
{\textstyle\prod_{j=1}^{m}}
\left[  \mathcal{L}^{\left(  j-1\right)  }\left(  y\right)  \right]
^{c_{i.j}}=%
{\textstyle\prod_{j=1}^{m+1}}
d_{i,j}\left[  \mathcal{L}^{\left(  j-1\right)  }\left(  y\right)  \right]
^{c_{i,j}^{\ast}},
\]
for some integers $d_{i,j},c_{i,j}^{\ast}\geq0.$ By replacing $d\ell
_{\gamma,J}^{\alpha}\left(  x\right)  /d\gamma$ by its formula $\left(
\ref{lc}\right)  ,$ we conclude that $\left(  \ref{latter-exp}\right)  $ is a
linear combinations of $E_{m+1,\alpha}^{\left(  i\right)  }\left(
x;\gamma\right)  ,$ which completes the proof of the first assertion. The
second one it obvious, it suffices to use assumption $\left[  A1\right]
-\left[  A3\right]  $ and the fact that the function $x\rightarrow\left(
x^{-a_{i}}\log^{b_{i}}x\right)  \ell_{\gamma,J}^{\alpha}\left(  x\right)  $ is
bounded over $x>1.$ Let us now show the third one. To do this end, we have to
make sure that $\int_{1}^{\infty}\left\vert E_{m,\alpha+1}^{\left(  i\right)
}\left(  x;\gamma\right)  \right\vert $ be finite. From assumption $\left[
A3\right]  ,$ the function $\mathcal{L}$ and their derivatives are bounded
over $\left(  0,1\right)  ,$ then $%
{\textstyle\prod_{j=1}^{m}}
\left[  \mathcal{L}^{\left(  j-1\right)  }\left(  x^{-1/\gamma}\right)
\right]  ^{c_{j}}$ is bounded over $x>1$ too, it follows that $E_{m,\alpha
+1}^{\left(  i\right)  }\left(  x;\gamma\right)  =O\left(  \left(  x^{-a_{i}%
}\log^{b_{i}}x\right)  \ell_{\gamma,J}^{\alpha+1}\left(  x\right)  \right)  .$
Since $\ell_{\gamma,J}^{\alpha+1}\left(  x\right)  =\ell_{\gamma,J}^{\alpha
}\left(  x\right)  \ell_{\gamma,J}\left(  x\right)  $ and $\ell_{\gamma
,J}^{\alpha}\left(  x\right)  =O\left(  1\right)  ,$ then $E_{m,\alpha
+1}^{\left(  i\right)  }\left(  x;\gamma\right)  =O\left(  \ell_{\gamma
,J}\left(  x\right)  \right)  .$ We have $\int_{1}^{\infty}\ell_{\gamma
,J}\left(  x\right)  ds=1,$ then $\int_{1}^{\infty}\left\vert E_{m,\alpha
+1}^{\left(  i\right)  }\left(  x;\gamma\right)  \right\vert dx$ is finite,
which completes the proof of the third assertion. The fourth assertion comes
by using similar arguments.
\end{proof}

\begin{proposition}
\textbf{\label{prop3bis}}Under assumption $\left[  A1\right]  -\left[
A2\right]  ,$ we have
\[
\int_{1}^{\infty}\left\vert \log^{b_{i}}x\frac{d}{dx}\ell_{\gamma,J}^{\alpha
}\left(  x\right)  \right\vert dx<\infty,\text{ for }\alpha>0.
\]
for every integer $b_{i}\geq0.$
\end{proposition}

\begin{proof}
We have $d\ell_{\gamma,J}^{\alpha}\left(  x\right)  /dx=\alpha\left(
d\ell_{\gamma,J}\left(  x\right)  /dx\right)  \ell_{\gamma,J}^{\alpha
-1}\left(  x\right)  $ and
\begin{align*}
\frac{d\ell_{\gamma,J}\left(  x\right)  }{dx}  &  =\frac{d}{dx}\left(
J\left(  x^{-1/\gamma}\right)  \ell_{\gamma}\left(  x\right)  \right) \\
&  =-\gamma^{-1}x^{-1/\gamma-1}J^{\left(  1\right)  }\left(  x^{-1/\gamma
}\right)  \ell_{\gamma}\left(  x\right)  +J\left(  x^{-1/\gamma}\right)
\frac{d}{dx}\ell_{\gamma}\left(  x\right)  .
\end{align*}
Form assumptions $\left[  A1\right]  -\left[  A2\right]  ,$ the function $J$
is nonincreasing then its derivative $J^{\left(  1\right)  }<0,$ on the other
hand $\ell_{\gamma}>0$ and $d\ell_{\gamma}\left(  x\right)  dx=-\gamma
^{-1}\left(  1+\gamma^{-1}\right)  x^{-2-1/\gamma}<0,$ it follows that
$\left\vert d\ell_{\gamma,J}^{\alpha}\left(  x\right)  /dx\right\vert $ is
less than of equal to the sum of
\[
M_{1}\left(  x\right)  :=\gamma^{-1}x^{-1/\gamma-1}J^{\left(  1\right)
}\left(  x^{-1/\gamma}\right)  \ell_{\gamma}\left(  x\right)  \text{ and
}M_{2}\left(  x\right)  :=J\left(  x^{-1/\gamma}\right)  \frac{d}{dx}%
\ell_{\gamma}\left(  x\right)  .\text{ }%
\]
The function $J^{\left(  1\right)  }$ being bounded, then $J^{\left(
1\right)  }\left(  x^{-1/\gamma}\right)  \ell_{\gamma}\left(  x\right)
=O\left(  1\right)  ,$ hence $\int_{1}^{\infty}M_{1}\left(  x\right)  dx$ is
finite. Likewise $M_{2}\left(  x\right)  =O\left(  d\ell_{\gamma}\left(
x\right)  /dx\right)  $ and $\int_{1}^{\infty}\left(  d\ell_{\gamma}\left(
x\right)  /dx\right)  dx=-\ell_{\gamma}\left(  1\right)  =-\gamma^{-1},$ it
follows that $\int_{1}^{\infty}M_{2}\left(  x\right)  dx$ is finite.
\end{proof}

\begin{proposition}
\textbf{\label{prop3}}Assume that\textbf{\ }assumption $\left[  A1\right]
-\left[  A3\right]  $ holds. Then for every $\alpha>0$ we have
\[
\int_{1}^{\infty}\left\vert \frac{d}{dx}\Psi_{\gamma,\alpha}^{\left(
m\right)  }\left(  x\right)  \right\vert dx<\infty,\text{ for }m=1,2.
\]

\end{proposition}

\begin{proof}
Since $\Psi_{\gamma,\alpha}^{\left(  m\right)  }\left(  x\right)  $ is linear
combination of $E_{m,\alpha}^{\left(  i\right)  }\left(  x;\gamma\right)  $
then it suffices to show that $\int_{1}^{\infty}\left\vert dE_{m,\alpha
}^{\left(  i\right)  }\left(  x;\gamma\right)  /dx\right\vert dx<\infty.$ It
is obvious that $dE_{m,\alpha}^{\left(  i\right)  }\left(  x;\gamma\right)
/dx $ is the sum of%
\[
\mathcal{D}_{1}\left(  x\right)  :=\ell_{\gamma,J}^{\alpha}\left(  x\right)
{\textstyle\prod_{j=1}^{m}}
\left[  \mathcal{L}^{\left(  j-1\right)  }\left(  x^{-1/\gamma}\right)
\right]  ^{c_{i,j}}\frac{d}{dx}\left(  x^{-a_{i}}\log^{b_{i}}x\right)  ,
\]%
\[
\mathcal{D}_{2}\left(  x\right)  :=\left(  x^{-a_{i}}\log^{b_{i}}x\right)
{\textstyle\prod_{j=1}^{m}}
\left[  \mathcal{L}^{\left(  j-1\right)  }\left(  x^{-1/\gamma}\right)
\right]  ^{c_{i,j}}\frac{d}{dx}\ell_{\gamma,J}^{\alpha}\left(  x\right)
\]
and%
\[
\mathcal{D}_{3}\left(  x\right)  :=\left(  x^{-a_{i}}\log^{b_{i}}x\right)
\ell_{\gamma,J}^{\alpha}\left(  x\right)  \frac{d}{dx}%
{\textstyle\prod_{j=1}^{m}}
\left[  \mathcal{L}^{\left(  j-1\right)  }\left(  x^{-1/\gamma}\right)
\right]  ^{c_{i,j}}.
\]
Observe that%
\[
\frac{d}{dx}\left(  x^{-a_{i}}\log^{b_{i}}x\right)  =\left\{
\begin{array}
[c]{cc}%
-a_{i}x^{-a_{i}-1} & \text{if }b_{i}=0,\\
\left(  b_{i}-a_{i}\log x\right)  x^{-a_{i}-1}\log^{b_{i}-1}x & \text{if
}b_{i}>1.
\end{array}
\right.  .
\]
In view of assumption $\left[  A3\right]  ,$ $%
{\textstyle\prod_{j=1}^{m}}
\left[  \mathcal{L}^{\left(  j-1\right)  }\left(  x^{-1/\gamma}\right)
\right]  ^{c_{i,j}}=O\left(  1\right)  ,$ over $x>1,$ on the other hand, $\ J$
is bounded and $\ell_{\gamma,J}^{\alpha}\left(  x\right)  =J^{\alpha}\left(
x^{-1/\gamma}\right)  \left(  \gamma^{-\alpha}x^{-\alpha\left(  1+1/\gamma
\right)  }\right)  ,$ then $\ell_{\gamma,J}^{\alpha}\left(  x\right)
=O\left(  x^{-\alpha\left(  1+1/\gamma\right)  }\right)  ,$ it follows that%
\[
\mathcal{D}_{1}\left(  x\right)  =O\left(  1\right)  x^{a_{i}-\alpha\left(
1+1/\gamma\right)  }\log^{b_{i}}x.
\]
It is easy to check that $\int_{1}^{\infty}x^{a_{i}-\alpha\left(
1+1/\gamma\right)  }\log^{b_{i}}xds<\infty,$ therefore $\int_{1}^{\infty
}\left\vert \mathcal{D}_{1}\left(  x\right)  \right\vert dx<\infty.$ Let us
now consider the second term and write $\mathcal{D}_{2}\left(  x\right)
=O\left(  \log^{b_{i}}x\frac{d}{dx}\ell_{\gamma,J}^{\alpha}\left(  x\right)
\right)  .$ From Proposition $\ref{prop3bis}$ we have$\int_{1}^{\infty
}\left\vert \log^{b_{i}}x\frac{d}{dx}\ell_{\gamma,J}^{\alpha}\left(  x\right)
\right\vert dx<\infty,$ which implies that $\int_{1}^{\infty}\left\vert
\mathcal{D}_{2}\left(  x\right)  \right\vert dx<\infty$ too. Next we show that
$\int_{1}^{\infty}\left\vert \mathcal{D}_{3}\left(  x\right)  \right\vert
dx<\infty$ as well. Indeed, observe that
\[
\frac{d}{dx}%
{\textstyle\prod_{j=1}^{m}}
\left[  \mathcal{L}^{\left(  j-1\right)  }\left(  x^{-1/\gamma}\right)
\right]  ^{c_{i,j}}=-\left(  1+\gamma^{-1}\right)  x^{-1/\gamma-1}%
{\textstyle\prod_{j=1}^{m+1}}
d_{i,j}^{\ast}\left[  \mathcal{L}^{\left(  j-1\right)  }\left(  x^{-1/\gamma
}\right)  \right]  ^{c_{i,j}^{\ast}},
\]
for some integers $c_{i,j}^{\ast},d_{i,j}^{\ast}\geq0.$ In view of assumptions
$\left[  A3\right]  ,$ we deduce that the latter quantity equals $O\left(
x^{-1/\gamma-1}\right)  ,$ over $x>1,$ then $\mathcal{D}_{2}\left(  x\right)
=O\left(  1\right)  x^{-1/\gamma-1}\left(  x^{-a_{i}}\log^{b_{i}}x\right)  .$
It is readily to check that $\int_{1}^{\infty}x^{-1/\gamma-1}\left(
x^{-a_{i}}\log^{b_{i}}x\right)  dx<\infty$ thereby $\int_{1}^{\infty
}\left\vert \mathcal{D}_{3}\left(  x\right)  \right\vert dx<\infty.$ In
conclusion $\int_{1}^{\infty}\left\vert dE_{m}^{\left(  i\right)  }\left(
x;\gamma\right)  /dx\right\vert dx$ is finite for $m=1,2,$ which completes the
proof of Proposition $\ref{prop3}.$
\end{proof}

\begin{proposition}
\textbf{\label{prop4}}In the probability space $\left(  \Omega,\mathcal{A}%
,\mathbf{P}\right)  ,$ there exists a standard Wiener process $\left\{
W\left(  x\right)  ,\text{ }x\geq0\right\}  $ such that for every
$0<\epsilon<1/2:$%
\[
\sup_{x\geq1}x^{\left(  1/2-\epsilon\right)  /\gamma_{0}}\left\vert
D_{k}\left(  x\right)  -W\left(  x^{-1/\gamma_{0}}\right)  \right\vert
\overset{\mathbf{P}}{\rightarrow}0,\text{ as }n\rightarrow\infty,
\]
where $D_{k}\left(  x\right)  :=\sqrt{k}\left(  \frac{n}{k}\overline{F}%
_{n}\left(  X_{n-k:n}x\right)  -\frac{n}{k}\overline{F}\left(  X_{n-k:n}%
x\right)  \right)  .$ Moreover
\[
\sup_{x\geq1}x^{\left(  \epsilon-1/2\right)  /\gamma_{0}}\left\vert
D_{k}\left(  x\right)  \right\vert =O_{\mathbf{P}}\left(  1\right)  ,\text{ as
}n\rightarrow\infty.
\]

\end{proposition}

\begin{proof}
It suffices to use Proposition 3.1 in \cite{EHL2006}, Potter's inequality
$\left(  \ref{first-order-PotterF}\right)  $ to $\overline{F}$ and fact that
$x^{\left(  \epsilon-1/2\right)  /\gamma_{0}}W\left(  x^{-1/\gamma_{0}%
}\right)  =O_{\mathbf{P}}\left(  1\right)  ,$ therefore we omit the details.
See for instance the Proof of Theorem 2.1 in \cite{BchMN16}.
\end{proof}

\newpage

\section{\textbf{Appendix B\label{sec8}}}%

\begin{figure}[ptbh]%
\centering
\includegraphics[
height=5.8081in,
width=5.8237in
]%
{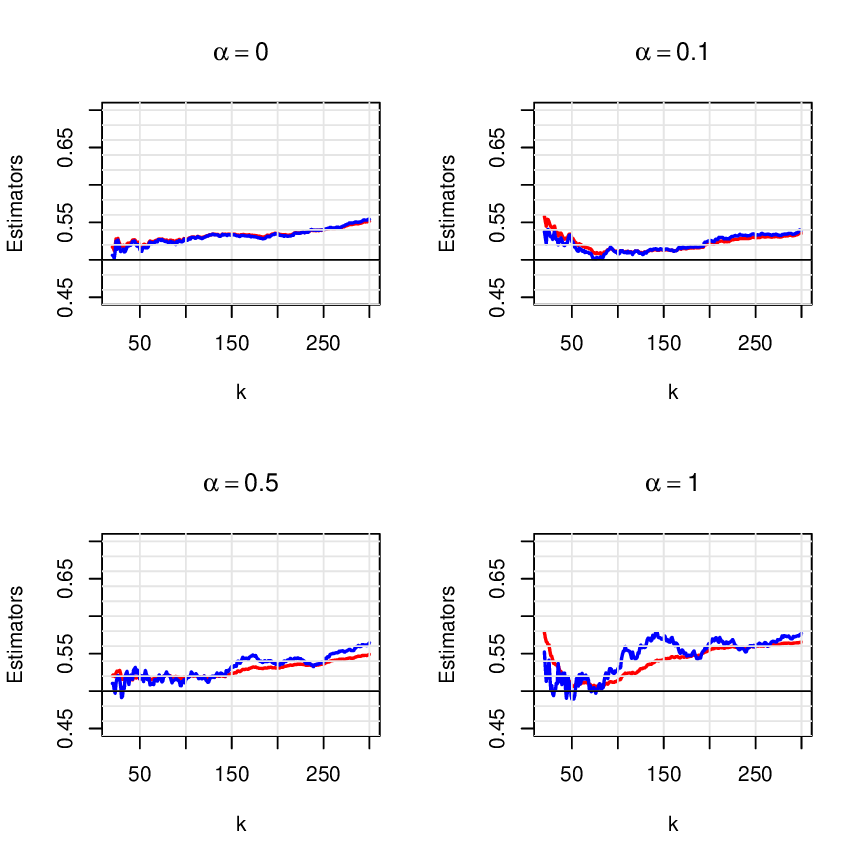}%
\caption{Plotting the estimators $\protect\widehat{\gamma}_{k,\alpha,J}$ (red
line) and $\protect\widehat{\gamma}_{k,\alpha}$ (blue line) for a Fr\'{e}chet
distribution with tail index: $\gamma=0.5$ and diffrent values of $\alpha,$
based on $2000$ samples of size $1000.$}%
\label{fig1}%
\end{figure}
\begin{figure}[ptbh]%
\centering
\includegraphics[
height=5.8081in,
width=5.8237in
]%
{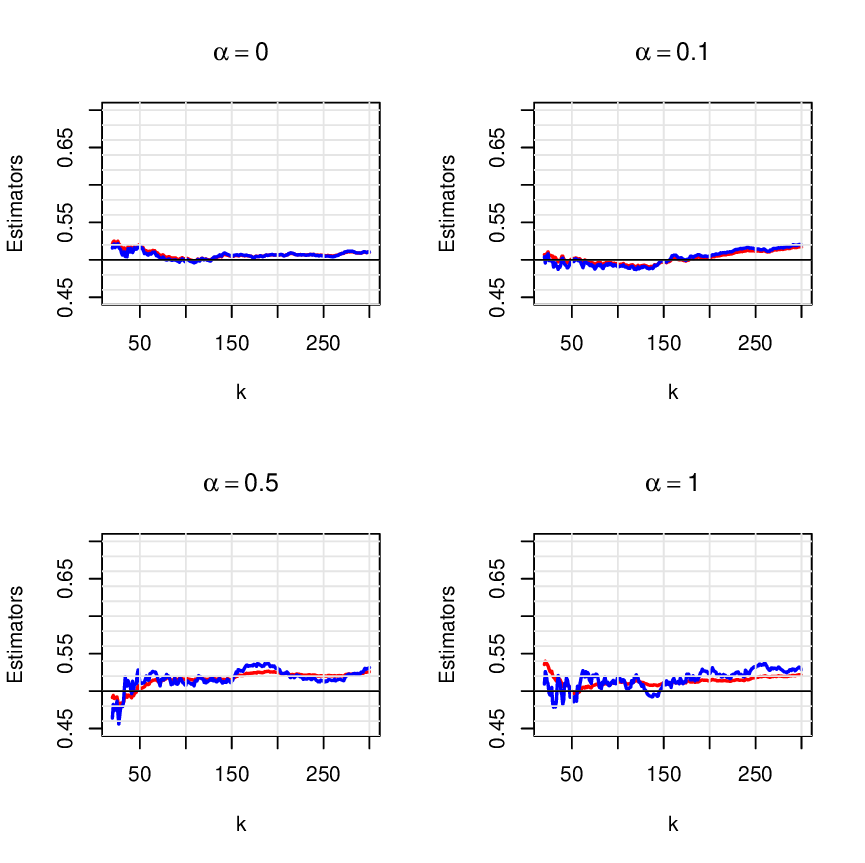}%
\caption{Plotting the estimators $\protect\widehat{\gamma}_{k,\alpha,J}$ (red
line) and $\protect\widehat{\gamma}_{k,\alpha}$ (blue line) for a Burr
distribution with tail index: $\gamma=0.5$ and diffrent values of $\alpha,$
based on $2000 $ samples of size $1000.$}%
\label{fig2}%
\end{figure}
%

\begin{figure}[ptbh]%
\centering
\includegraphics[
height=5.8081in,
width=5.8237in
]%
{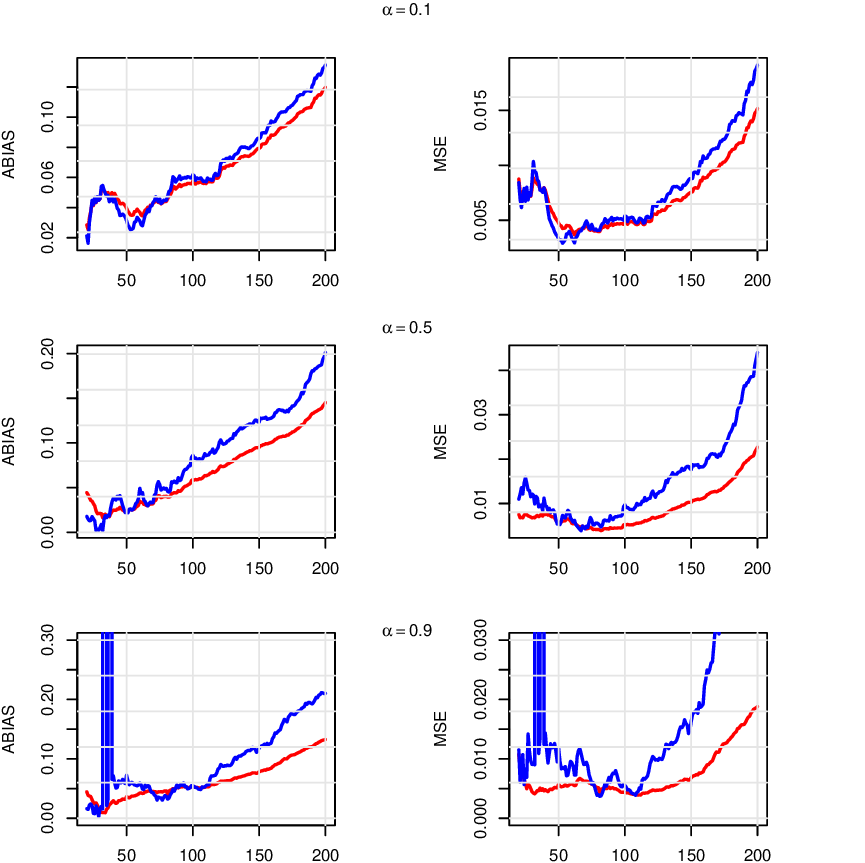}%
\caption{Absolute bias (left panel) and MSE (right panel) of
$\protect\widehat{\gamma}_{k,\alpha,J} $ (red) and $\protect\widehat{\gamma
}_{k,\alpha}$ (blue), corresponding to Frechet distribution with tail index:
$\gamma=0.4$ and diffrent values of $\alpha,$ based on $2000$ samples of size
$300.$}%
\label{fig3}%
\end{figure}
%

\begin{figure}[ptb]%
\centering
\includegraphics[
height=5.8081in,
width=5.8237in
]%
{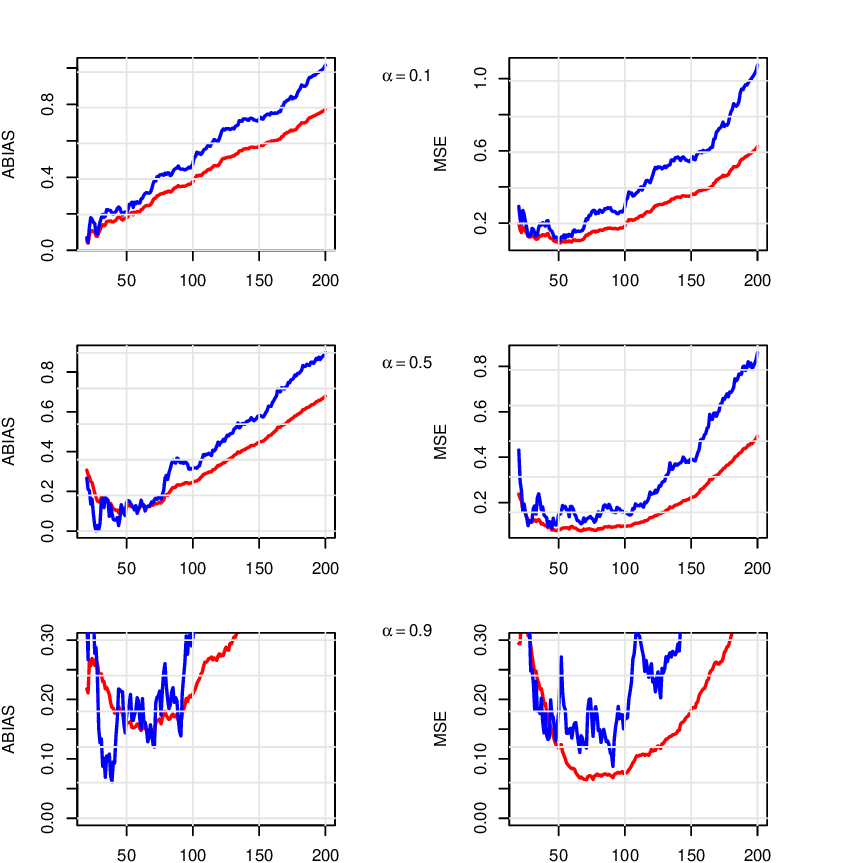}%
\caption{Absolute bias (left panel) and MSE (right panel) of
$\protect\widehat{\gamma}_{k,\alpha,J} $ (red) and $\protect\widehat{\gamma
}_{k,\alpha}$ (blue), corresponding to Frechet distribution with tail index:
$\gamma=1.5$ and diffrent values of $\alpha,$ based on $2000$ samples of size
$300.$}%
\label{fig4}%
\end{figure}
%

\begin{figure}[ptb]%
\centering
\includegraphics[
height=5.8081in,
width=5.8237in
]%
{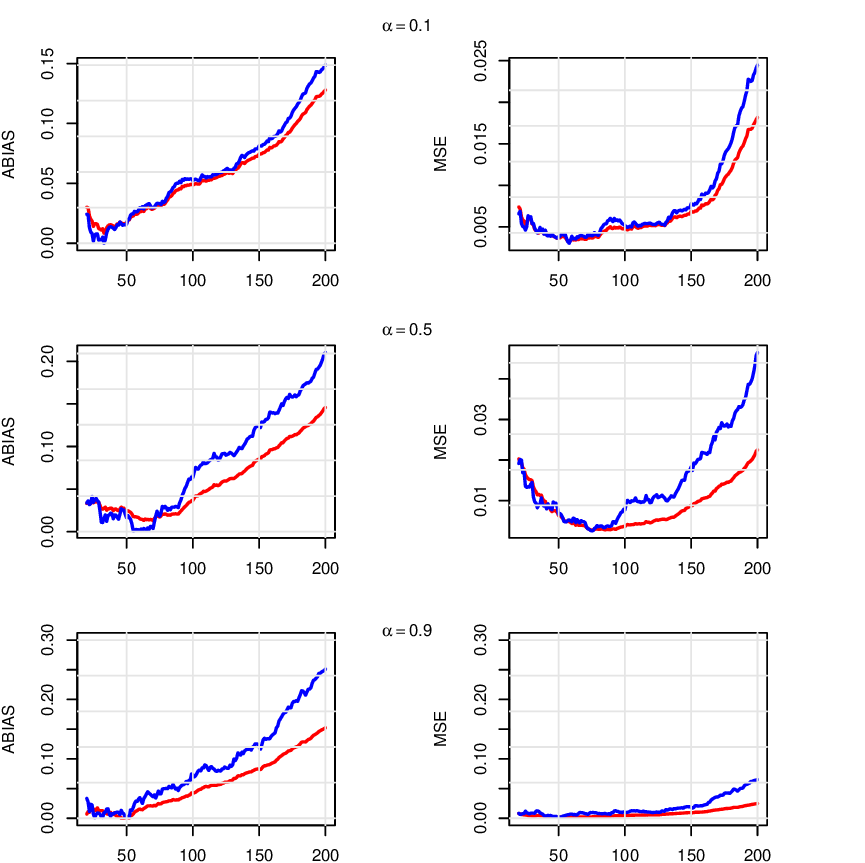}%
\caption{Absolute bias (left panel) and MSE (right panel) of
$\protect\widehat{\gamma}_{k,\alpha,J} $ (red) and $\protect\widehat{\gamma
}_{k,\alpha}$ (blue), corresponding to Burr distribution with tail index:
$\gamma=0.4$ and diffrent values of $\alpha,$ based on $2000$ samples of size
$300.$}%
\label{fig5}%
\end{figure}
\begin{figure}[ptb]%
\centering
\includegraphics[
height=5.8081in,
width=5.8237in
]%
{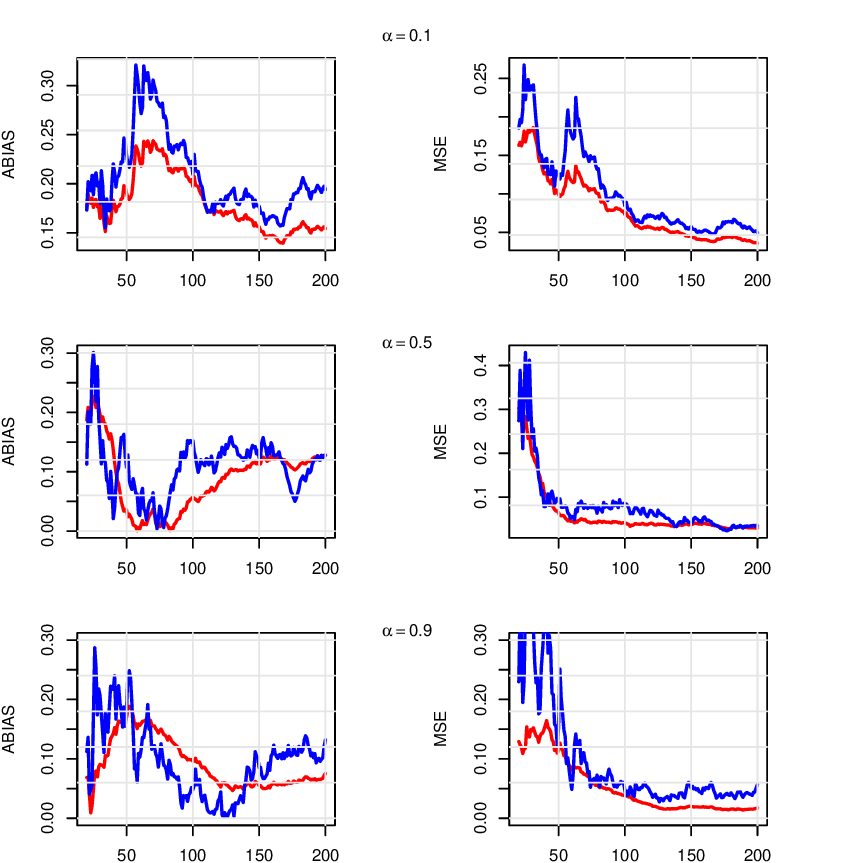}%
\caption{Absolute bias (left panel) and MSE (right panel) of
$\protect\widehat{\gamma}_{k,\alpha,J}$ (red) and $\protect\widehat{\gamma
}_{k,\alpha}$ (blue), corresponding to Burr distribution with tail index:
$\gamma=1.5$ and diffrent values of $\alpha,$ based on $2000$ samples of size
$300.$}%
\label{fig6}%
\end{figure}
%

\begin{figure}[ptb]%
\centering
\includegraphics[
trim=0.000000in 0.000000in -0.498951in -0.504915in,
height=5.8115in,
width=5.8219in
]%
{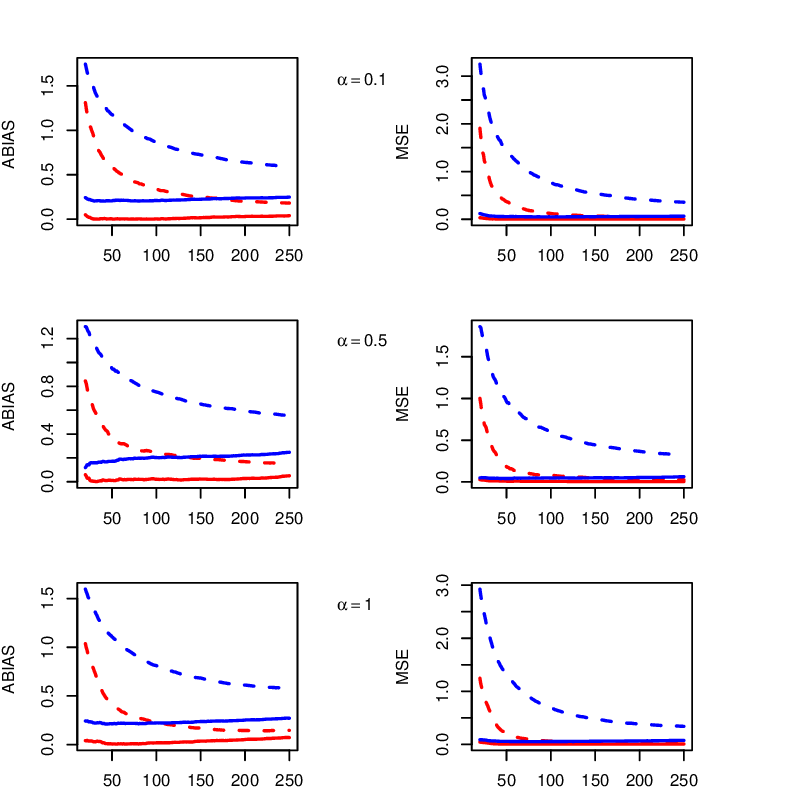}%
\caption{$S1:$ Comparaison in terms of absolute bias (left panel) and MSE
(right panel) of the two estimators $\protect\widehat{\gamma}_{k,\alpha,J}$ (
red) and $\protect\widehat{\gamma}_{k,J}$\ (blue)\ in the both cases when the
estimators are pure (solid line) and 0.1-contaminated (dashed line),
corresponding to different values of $\alpha,$ based on $2000$\ samples of
size $500. $}%
\label{fig7}%
\end{figure}
\ \
\begin{figure}[ptb]%
\centering
\includegraphics[
trim=0.020000in 0.019945in -0.520004in -0.524860in,
height=5.8115in,
width=5.8219in
]%
{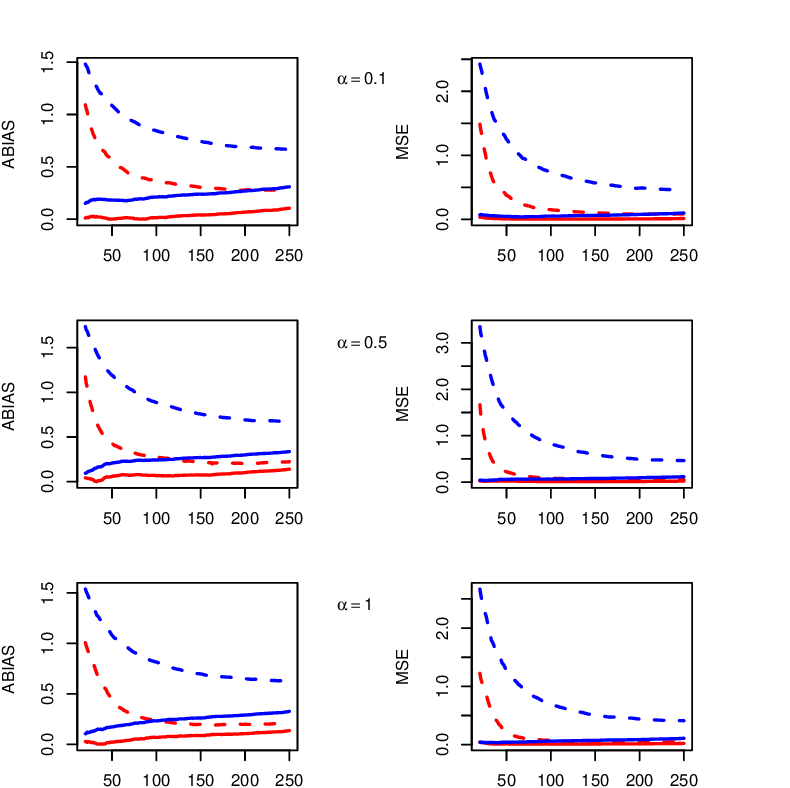}%
\caption{$S2:$ Comparaison in terms of absolute bias (left panel) and MSE
(right panel) of the two estimators $\protect\widehat{\gamma}_{k,\alpha,J}$ (
red) and $\protect\widehat{\gamma}_{k,J}$\ (blue)\ in the both cases when the
estimators are pure (solid line) and 0.1-contaminated (dashed line),
corresponding to different values of $\alpha,$ based on $2000$\ samples of
size $500.$}%
\label{fig8}%
\end{figure}
%

\begin{figure}[ptb]%
\centering
\includegraphics[
trim=0.000000in 0.000000in -0.499478in -0.504390in,
height=5.8115in,
width=5.8219in
]%
{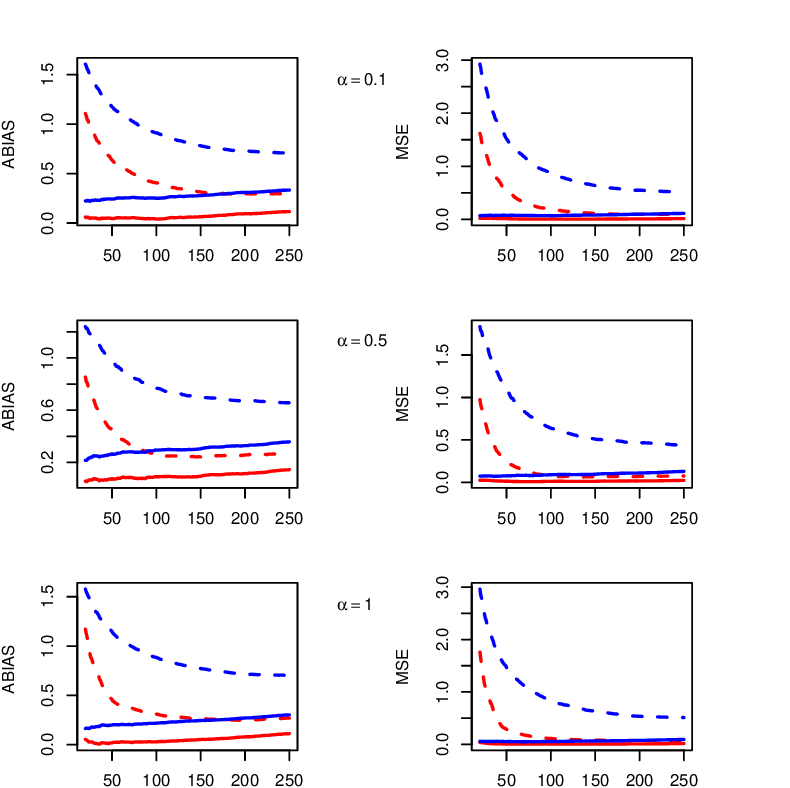}%
\caption{$S3:$ Comparaison in terms of absolute bias (left panel) and MSE
(right panel) of the two estimators $\protect\widehat{\gamma}_{k,\alpha,J}$ (
red) and $\protect\widehat{\gamma}_{k,J}$\ (blue)\ in the both cases when the
estimators are pure (solid line) and 0.1-contaminated (dashed line),
corresponding to different values of $\alpha,$ based on $2000$\ samples of
size $500. $}%
\label{fig9}%
\end{figure}
%

\begin{figure}[ptb]%
\centering
\includegraphics[
trim=0.000000in 0.000000in -0.499478in -0.504390in,
height=5.8115in,
width=5.8219in
]%
{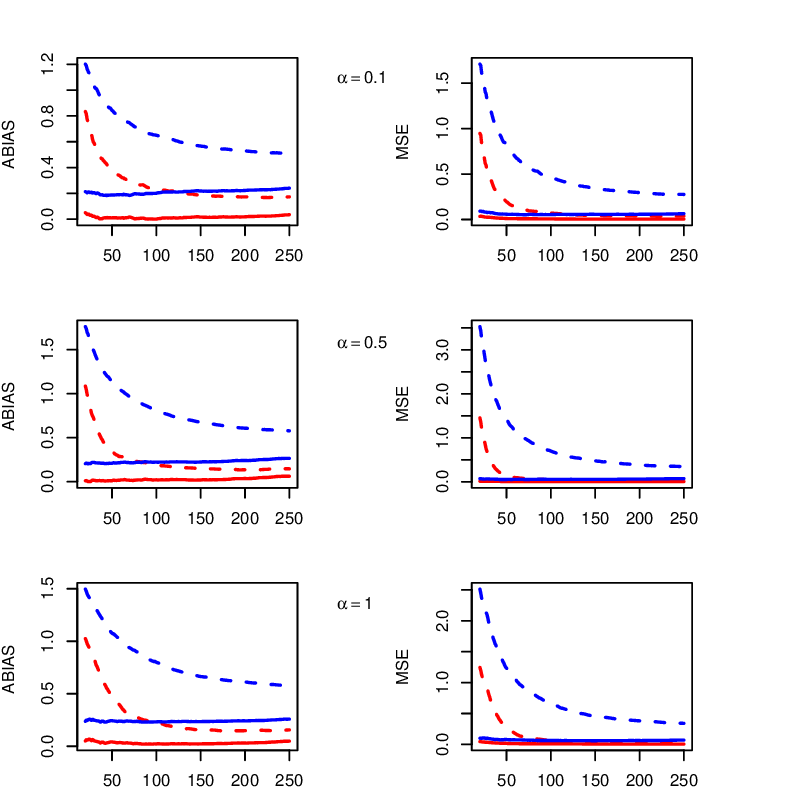}%
\caption{$S4:$ Comparaison in terms of absolute bias (left panel) and MSE
(right panel) of the two estimators $\protect\widehat{\gamma}_{k,\alpha,J}$ (
red) and $\protect\widehat{\gamma}_{k,J}$\ (blue)\ in the both cases when the
estimators are pure (solid line) and 0.1-contaminated (dashed line),
corresponding to different values of $\alpha$, based on $2000$\ samples of
size $500. $}%
\label{fig10}%
\end{figure}


\begin{thebibliography}{99999999999999999999999999999999999999}                                                           %


\bibitem[Basu \textit{et al.}(1998)]{Basu98}Basu, A., Harris, I. R., Hjort, N.
L., \& Jones, M. C. 1998. Robust and efficient estimation by minimizing a
density power divergence. \textit{Biometrika} 85, no. 3, 549--559.

\bibitem[Beirlant \textit{et al}.(1996)]{BVT96}Beirlant, J., Vynckier, P.,
Teugels, \& Jozef L. 1996. Tail index estimation, Pareto quantile plots, and
regression diagnostics. \textit{J. Amer. Statist. Assoc.} 91, 1659--1667.

\bibitem[Benchaira \textit{et al}.(2016)]{BchMN16}Benchaira, S., Meraghni, D.,
\& Necir, A. 2016. Tail product-limit process for truncated data with
application to extreme value index estimation. \textit{Extremes, }19, 219-251.

\bibitem[Brahimi \textit{et} \textit{al.}(2015)]{BMN-2015}Brahimi, B.,
Meraghni, D., \& Necir, A., 2015. Approximations to the tail index estimator
of a heavy-tailed distribution under random censoring and application.
\textit{Math. Methods Statist.} 24\textbf{,} 266-279.

\bibitem[Brazauskas and Serfling(2000)]{BS2000}Brazauskas, V., \& Serfling, R.
2000. Robust and efficient estimation of the tail index of a single-parameter
Pareto distribution. \textit{North American Actuarial Journal,} 4, 12-27.

\bibitem[Caeiro \textit{et al.}(2021)]{CMS21}Caeiro, F., Mateus, A., \&
Soltane, L. 2021. A class of weighted Hill estimators. Comput. Math. Methods
3, Paper No. e1167, 12 pp.

\bibitem[Cs\"{o}rg\H{o} et \textit{al}.(1985)]{CDM85}Cs\"{o}rg\H{o}, S.,
Deheuvels, P., \& Mason, D. 1985. Kernel estimates of the tail index of a
distribution. \textit{Ann. Statist.} 13, 1050--1077.

\bibitem[Dell'Aquila and Embrechts(2006)]{DE2006}Dell'Aquila, R., \&
Embrechts, P. 2006. Extremes and robustness: A contradiction?
\textit{Financial Markets and Portfolio Management,} 20, 103--18. doi:10.1007/s11408-006-0002-x.

\bibitem[Dierchx \textit{et al.}(2013)]{DGG13}Dierckx, G., Goegebeur, Y., \&
Guillou, A. 2013. An asymptotically unbiased minimum density power divergence
estimator for the Pareto-tail index.\textit{\ Journal of Multivariate
Analysis,} 121, 70--86.

\bibitem[Dierckx \textit{et al.}(2021)]{DGG2021}Dierckx, G., Goegebeur, Y., \&
Guillou, A. 2021. Local robust estimation of Pareto-type tails with random
right censoring. \textit{Sankhya: The Indian Journal of Statistics,} 83, 70--108.

\bibitem[Einmahl \textit{et al.}(2006)]{EHL2006}Einmahl, J. H. J., de Haan,
L., \& Li, D. 2006. Weighted approximations of tail copula processes with
application to testing the bivariate extreme value condition. \textit{Ann.
Statist.} 34, 1987--2014.

\bibitem[Ghosh and Basu(2013)]{GB2013}Ghosh, A., \& Basu, A. 2013. Robust
estimation for independent non-homogeneous observations using density power
divergence with applications to linear regression.\textit{\ Electronic Journal
of Statistics,} 7, 2420--2456.

\bibitem[Ghosh(2017)]{Ghosh2017}Ghosh, A. 2017. Divergence based robust
estimation of the tail index through an exponential regression model.
\textit{Statistical Methods \& Applications,} 26, 181--213.

\bibitem[Goregebeur \textit{et al}.(2010)]{GBW10}Goegebeur, Y., Beirlant, J.,
\& De Wet, T. 2010. Kernel estimators for the second order parameter in
extreme value statistics. \textit{J. Statist. Plan. Inference. }140, 2632--2652.

\bibitem[Goegebeur\ \textit{et al.}(2014)]{GGV2014}Goegebeur, Y., Guillou, A.,
\& Verster, A. 2014. Robust and asymptotically unbiased estimation of extreme
quantiles for heavy tailed distributions.\textit{\ Statistics \& Probability
Letters,} 87, 108--114.

\bibitem[Groeneboom \textit{et al}.(2003)]{GLW2003}Groeneboom, P.,
Lopuha\"{a}, H. P., \& de Wolf, P. P. 2003. Kernel-type estimators for the
extreme value index. \textit{Ann. Statist.} 31, 1956--1995.

\bibitem[de Haan and Stadtm\"{u}ller(1996)]{deHS96}de Haan, L., \&
Stadtm\"{u}ller, U. 1996. Generalized regular variation of second order.
\textit{J. Australian Math. Soc.} (Series A) 61, 381--395.

\bibitem[de Haan and Ferreira(2006)]{deHF06}de Haan, L., \& Ferreira, A. 2006.
\textit{Extreme Value Theory: An Introduction}. \textit{Springer}.

\bibitem[Hill(1975)]{Hill75}Hill, B.M. 1975. A simple general approach to
inference about the tail of a distribution. \textit{Ann. Statist.} 3, 1163--1174.

\bibitem[Hua and Joe(2011)]{HJ11}Hua, L., \& Joe, H. 2011. Second order
regular variation and conditional tail expectation of multiple
risks.\textit{\ Insurance Math. Econom.} \textbf{49}, 537--546.

\bibitem[H\"{u}sler \textit{et al}.(2006)]{HLM2006}H\"{u}sler, J., Li, D., \&
M\"{u}ller, S. 2006. Weighted least squares estimation of the extreme value
index. \textit{Stat Probab Lett.} 76, 920--930.

\bibitem[Ju\'{a}r\`{e}z and Schucany(2004)]{JS2004}Ju\'{a}r\`{e}z, S. F., \&
Schucany, W. R. 2004. Robust and efficient estimation for the generalized
Pareto distribution. \textit{Extremes,} 7, 237--51.

\bibitem[Kim and Lee(2008)]{Kim2008}Kim, M., \& Lee, S. 2008. Estimation of a
tail index based on minimum density power divergence. \textit{Journal of
Multivariate Analysis,} 99, 2453--2471.

\bibitem[Lehmann and Casella(1998)]{LC98}Lehmann, E. L., \& Casella, G. 1998.
Theory of point estimation. \textit{Springer.}

\bibitem[Mattys and Beirlant(2003)]{MB2003}Matthys, G., \& Beirlant, J. 2003.
Estimating the extreme value index and high quantiles with exponential
regression models. \textit{Statistica Sinica,} 13, 853--880.

\bibitem[Minkah \textit{el al.}(2023)]{MWG2023}Minkah, R., de Wet, T., \&
Ghosh, A. 2023. Robust estimation of Pareto-type tail index through an
exponential regression model. \textit{Comm. Statist. Theory Methods,} 52, 479-498.

\bibitem[Viharos(1999)]{V99}Viharos, L., 1999. Weighted least-squares
estimators of tail indices. \textit{Probab. Math. Statist}. 19, Acta Univ.
Wratislav. No. 2198, 249--265.
\end{thebibliography}
\end{document}